\documentclass[a4paper,reqno,11pt]{amsart}
\usepackage[utf8]{inputenc}  

\usepackage{float} 
\usepackage[in]{fullpage}

\newcommand{\ds}{\displaystyle}

\newcommand{\tensor}{\otimes}

\newcommand{\op}{\mathcal}

\newcommand{\cdc}{,\dots,}

\newcommand{\FT}{\mathsf{ft}}
\newcommand{\FTGK}{\mathsf{FT}}

\newcommand{\DM}{\mathsf{FT}(H_\ast(\overline{\op{M}}))}

\usepackage{amsmath}%
\usepackage{amsthm}
\usepackage{amsfonts}%
\usepackage{amssymb}%
\usepackage{graphicx}
\usepackage{xy,amsthm,enumerate,xypic,array}  
\usepackage{xcolor}

\input{xy}
\xyoption{all}

\numberwithin{equation}{section}

\newtheorem{theorem}{Theorem}[section]
\theoremstyle{plain}

\newtheorem{corollary}[theorem]{Corollary}
\newtheorem{lemma}[theorem]{Lemma}

\theoremstyle{definition}
\newtheorem{definition}[theorem]{Definition}
\newtheorem{example}[theorem]{Example}

\newtheorem{remark}[theorem]{Remark}

\newtheorem{construction}[theorem]{Construction}
\allowdisplaybreaks[2]

\addtocounter{MaxMatrixCols}{2}

\begin{document}
\title{Toward a minimal model for $H_\ast(\overline{\op{M}})$} 
\author{Ben C. Ward}
\email{benward@bgsu.edu}

\maketitle
\begin{abstract} 
The modular operad $H_\ast(\overline{\op{M}}_{g,n})$ of the homology of Deligne-Mumford compactifications of moduli spaces of pointed Riemann surfaces has a minimal model governed by higher homology operations on the open moduli spaces $H_\ast(\op{M}_{g,n})$.  Using Getzler's elliptic relation, we give an explicit construction of the first family of such higher operations.

\end{abstract}

\section{Introduction}

Let $\overline{\op{M}}_{g,n}$ denote the Deligne-Mumford compactification of $\op{M}_{g,n}$, the moduli space of genus $g$ Riemann surfaces with $n$ marked points, regarded as smooth complex orbifolds.  Gluing marked points endows the collection $\overline{\op{M}}_{g,n}$ with the structure of a modular operad \cite{GeK2}.  We write $H_\ast(\overline{\op{M}}_{g,n})$ for the rational homology of the underlying spaces; this collection of graded $S_n$-modules carries the induced structure of a graded modular operad which we denote $H_\ast(\overline{\op{M}})$.

Write $\mathfrak{p}H_\ast(\op{M})$ for the collection of graded $S_n$-representations $H_
{\ast-6g-2n+6}(\op{M}_{g,n}, \mathbb{Q})$.  Such a collection is called a stable $\mathbb{S}$-module \cite{GeK2}.
If $A$ is a modular operad we write $\FTGK(A)$ for its Feynman transform, regarded as a $\mathfrak{K}$-twisted modular operad, and in particular a differential graded stable $\mathbb{S}$-module.
The following result follows from \cite[Proposition 6.11]{GeK2} and formality of $H_\ast(\overline{\op{M}})$, see also \cite[Theorem 28]{AWZ}.

\begin{theorem}\label{thm1} \cite{GeK2} The differential graded stable $\mathbb{S}$-modules $\FTGK(H_\ast(\overline{\op{M}}))$ and $\mathfrak{p}H_\ast(\op{M})$ are quasi-isomorphic.
\end{theorem}
 
If we regard $\mathfrak{p}H_\ast(\op{M})$ merely as a graded $\mathbb{S}$-module, we may ask the question: does there exist a $\mathfrak{K}$-modular operad structure on $\mathfrak{p}H_\ast(\op{M})$ such that Theorem $\ref{thm1}$ may be upgraded to a quasi-isomorphism of modular operads?  The answer to this question is no; it follows from an analysis of the weight filtration on $H^\ast(\op{M}_{1,n})$ \cite{PetG1}; see \cite[Proposition 1.11]{AP}.

\begin{theorem}\label{thm2}\cite{AP} There does not exist a $\mathfrak{K}$-modular operad structure on $\mathfrak{p}H_\ast(\op{M})$ for which $\FTGK(H_\ast(\overline{\op{M}}))$ and $\mathfrak{p}H_\ast(\op{M})$ are quasi-isomorphic as $\mathfrak{K}$-modular operads.  
\end{theorem}

From the perspective of operadic algebra the notion of a modular operad is homotopically inflexible.  In \cite{WardMP} a homotopy invariant generalization of modular operads, called weak modular operads, were introduced.  Roughly speaking, weak modular operads are to modular opeards as $A_\infty$-algebras are to associative algebras.  Combining the suitable adaptation of homotopy transfer theory in loc.cit.\ with Theorem $\ref{thm1}$ rectifies the obstructions presented by Theorem $\ref{thm2}$ to prove:

\begin{theorem}\label{thm3}
	There exists a weak $\mathfrak{K}$-modular operad structure on the stable $\mathbb{S}$-module $\mathfrak{p}H_\ast(\op{M})$ for which there exists an $\infty$-quasi-isomorphism of weak $\mathfrak{K}$-modular operads
	\begin{equation*}
	\DM\stackrel{\sim}\rightsquigarrow \mathfrak{p}H_\ast(\op{M}).
	\end{equation*}
\end{theorem}
This structure of a weak $\mathfrak{K}$-modular operad is manifest as a family of homology operations 
\begin{equation*}
H_n(\op{M})(\gamma)_{Aut(\gamma)}\to H_{n+2|\gamma|-1}(\op{M}_{g,n})
\end{equation*}
indexed by stable graphs $\gamma$ having $|\gamma|$ edges.  We call these operations (operadic) Massey products.  We call such a Massey product a ``higher operation'' if its corresponding graph has more than one edge.

Let us denote this weak $\mathfrak{K}$-modular operad by $\mathfrak{p}h(\op{M})$.  Such a structure is not unique, but any two such structures are related by an $\infty$-isomorphism.  The Feynman transform $\FTGK$ may be extended to weak modular operads; we denote this extension by $\FT$. We prove:

\begin{theorem}\label{thm4}  The modular operad	$\FT(\mathfrak{p}h(\op{M}))$ is the minimal model for $H_\ast(\overline{\op{M}})$.
\end{theorem}
Our notion of minimal modular operad comes from \cite{GNPR}, namely we prove that the canonical tower of modular operads associated to the weak Feynman transform of a $\mathfrak{K}$-modular operad is a sequence of principal extensions.  The idea to construct a minimal model for $H_\ast(\overline{\op{M}})$ in such a way is suggested in \cite{WardMP} and in \cite[Remark 1.22]{DSVV}.

The first instance where such higher Massey products are unavoidable is the case $(g,n) = (1,4)$.  
This fact is reflected on the one hand by the existence of a relation between the codimension 2 strata in $H_4(\overline{\op{M}}_{1,4})$, after \cite{GetGW}, and on the other hand by impurity of the mixed Hodge structure on $H^\ast(\op{M}_{1,4})$, after \cite{PetG1}.  Comparing the natural bigrading on $\DM$ with the bigrading on $H_\ast(\op{M}_{1,4})$ induced by the weight filtration,
 the source and target of any such higher operation
 are each supported on a 1-dimensional space, and so such a Massey product is specified by a number.    If we restrict attention to $(g,n)\leq (1,4)$, ordered lexicographically, the following result gives the explicit form of all possible such higher operations:

\begin{theorem}\label{main}
Fix $\mathsf{e},\mathsf{w} \in \mathbb{Q}$ with $\mathsf{e}\neq 0$.  The nine homology operations
\begin{equation*}
\mu_\bullet \colon H_0(\op{M})(\delta_\bullet) \to \mathfrak{gr}_{4}^{W}(H_3(\op{M}_{1,4}))
\end{equation*}
corresponding to the nine stable graphs $\delta_\bullet$ of type (1,4) encoded by the table:
\begin{equation*}
	\begin{tabular}{c|c|c|c|c|c|c|c|c|c}
Massey Product	$\mu_{\bullet}$ &	$\mu_{2,2}$ & $\mu_{2,3}$ & $\mu_{3,4}$ & $\mu_{2,4}$ & $\mu_{0,4}$& $\mu_{0,3}$ & $\mu_{0,2}$ & 	$\mu_{\alpha}$ & $\mu_{\beta}$ \\ \hline
coefficient & $12\mathsf{e}$	 & $-4\mathsf{e}$ & $6\mathsf{e}$	 & $-2\mathsf{e}$	 & $\mathsf{e}+6\mathsf{w}$	 & $\mathsf{e}+3\mathsf{w}$ & $\mathsf{w}$	& $-3\mathsf{w}$  &	$-2\mathsf{e}-4\mathsf{w}$ \\
\end{tabular} 
\end{equation*}
 endow $\mathfrak{p}H_\ast(\op{M})$ with the structure of a weak $\mathfrak{K}$-modular operad for which there exists an explicitly presented $\infty$-quasi-isomorphism of (1,4)-truncated $\mathfrak{K}$-modular operads
	\begin{equation*}
\DM\stackrel{\sim}\rightsquigarrow \mathfrak{p}H_\ast(\op{M}).
\end{equation*}
\end{theorem}  Here, our notation for graphs follows \cite{GetGW}; see Figure $\ref{fig:14graphs}$ below.  Our notion of truncated modular operads is a straight-forward generalization of the notion of the same name given in \cite{GNPR}, see Section $\ref{sec:truncated}$.

We emphasize that as part of the proof of Theorem $\ref{main}$ the $\infty$-quasi-isomorphism in question is explicitly constructed.  An $\infty$-isomorphism relating different choices of $(\mathsf{e},\mathsf{w})$ is also explicitly constructed using the gravity relations of \cite{Getgrav}; see Lemma $\ref{unique}$.  Our analysis (Section $\ref{dmsec}$) also gives representing cycles for each homology class in $\op{M}_{1,4}$.  Some of this combinatorial analysis can be generalized to arbitrary $(g,n)$, for example to glean information about representing (co)cycles for
$\mathfrak{gr}^W_{6g+2n-8} H^{3g+n-4}(\op{M}_{g,n})$, which may be of independent interest.  See e.g.\ Lemma $\ref{row2lem}$.

The form of these higher operations can be given the following {\it a posteriori} explanation.  The minimal model $\FT(h_\ast(\op{M}))$ contains a one-dimensional subspace corresponding to the corolla labeled by $\mathfrak{gr}_4^WH^3(\op{M}_{1,4})$.  The weak Feynman transform differential composed with the weak equivalence $\FT(h_\ast(\op{M}))\to H_\ast(\overline{\op{M}}_{1,4})$ must vanish for this map to be dg.  Here, this weak equivalence is induced by the isomorphisms $H_0(\op{M}_{g,n})^\ast\cong H_{6g+2n-6}(\overline{\op{M}}_{g,n})$, and so the sum of the linear duals of Massey products landing in internal degree $0$ must correspond to a relation among boundary cycles in $H_{4}(\overline{\op{M}}_{1,4})$.  These relations were explicitly computed by Getzler \cite{GetGW}, and the sum of the linear duals of the Massey products given in Theorem $\ref{main}$ corresponds to $\mathsf{e}$ copies of the Getzler's elliptic relation along with $\mathsf{w}$ copies of a manifestation of the WDVV relations.

\tableofcontents

\section{Weak modular operads, the Feynman transform and minimality.}
In this section we sketch details pertaining to graphs, modular operads and the Feynman transform.  We refer to \cite{GeK2} for full details.  We also discuss weak modular operads and their $\infty$-morphisms; see \cite{WardMP} for more details.  Throughout we use homological degree conventions, denote the symmetric group by $S_n$ and denote the (graded) linear dual of a (graded) vector space $V$ by $V^\ast$.  We denote shift operators for graded vectors spaces by $\Sigma^r$, with $(\Sigma^rV)_{r+s} = V_s$.

\subsection{Modular graphs}
A graph $\gamma = (H,V,a,i)$ is a set of half edges $H$ and a non-empty set of vertices $V$, equipped with a function $a\colon H\to V$ assigning a vertex to each half edge and an involution $i\colon H\to H$ on the set of half edges.  Orbits of the involution of size two are called edges; orbits of size one are called legs.  A graph with no edges is called a corolla.  The valence of a vertex $v$ is $|v|:= |a^{-1}(v)|$.

A graph may be represented as a 1-dimensional CW complex.  A graph is said to be connected if this CW complex is connected.  The genus of a graph is the first betti number of this CW complex and is denoted $\beta_1(\gamma)$.  A genus labeling of a graph is a function $g\colon V\to \mathbb{Z}_{\geq 0}$.    The total genus of a genus labeled graph is defined to be $g(\gamma):= \beta_1(\gamma)+\sum_{v\in V}g(v)$.  A connected, genus labeled graph is stable if $|v|+2g(v)\geq 3$ for each vertex $v\in V$.  A stable graph is $\gamma$ is said to be of type $\ast_\gamma = (g,n)$ if it has total genus $g$ and $n$ legs.

A modular graph is a stable graph along with a bijection between the set of legs and some $\{1\cdc n\}$, called the leg labeling.  An isomorphism of modular graphs is a pair of maps of vertices and half edges which respect $a$ and $i$, as well as the genus labeling and leg labeling. Since we will typically be interested in modular graphs only up to such label preserving isomorphisms, let's define $\Gamma(g,n)^{S_n}$ (resp.\ $\Gamma(g,n)$) to be a skeletal set of stable graphs (resp.\ modular graphs) of type $(g,n)$.  We write $\Gamma(g,n)^{S_n}_r$ (resp.\ $\Gamma(g,n)_r$) for the subsets of such graphs having $r$ edges.

\subsection{Modular operads}  A stable $\mathbb{S}$-module is a collection of dg $S_n$-modules $A(g,n)$ for each pair of non-negative integers with $2g+n\geq 3$. Given a stable $\mathbb{S}$-module $A$ and a set $X$ with $n$ elements, we define $A(g,X)$ as the left Kan extension of $A(g,-)$ along the inclusion functor from the category with one object $\{1\cdc n\}$ and morphisms $S_n$ to the category of all finite sets and bijections.  Given a stable graph $\gamma$, we then define $A(v) = A(g(v),a^{-1}(v))$ for each $v\in V(\gamma)$.  We then define $A(\gamma) = \tensor_{v\in V(\gamma)}A(v)$.

A modular operad is a stable $\mathbb{S}$-module along with associative composition maps $\mu_\gamma\colon  A(\gamma) \to A(g,n)$ which are $S_n$-equivariant and which are coinvariant with respect to isomorphisms of graphs.  In particular it is sufficient to specify an operation $ A(\gamma)_{Aut(\gamma)} \to A(g,n)$ for each $\gamma \in \Gamma(g,n)$. 
We call this map contraction of the graph $\gamma$.

The associativity condition governs compositions of contractions.  Given a graph $\gamma$ and a connected subgraph $\gamma^\prime$, one may collapse the edges of $\gamma^\prime$ to form a new stable graph $\gamma/\gamma^\prime$.  Upon passing to coinvariants by automorphisms, it is possible to compose $\mu_{\gamma/\gamma^\prime}$ and $\mu_{\gamma^\prime}$, and the associativity axiom for modular operads states that $\mu_\gamma =\mu_{\gamma/\gamma^\prime} \circ \mu_{\gamma^\prime}$.  This associativity condition may be understood informally by saying that contracting a graph all at once, or one edge at time, or via any ordered partition of the set of edges results in the same map.  In particular, in a modular operad, it is enough to specify the one edged compositions.

Let us also recall the analogous notion of a $\mathfrak{K}$-modular operad.  For a stable graph $\gamma$ we define $\mathfrak{K}(\gamma)$ to be the top exterior power of the set of edges of $\gamma$, shifted to be concentrated in degree -$|$edges of $\gamma|$.  A $\mathfrak{K}$-modular operad is a stable $\mathbb{S}$-module along with associative contraction maps $A(\gamma) \tensor_{Aut(\gamma)} \mathfrak{K}(\gamma) \to A(g,n)$ for each $\gamma \in \Gamma(g,n)$.

\subsection{Weak modular operads}

A nest on a modular graph is a connected subgraph with no legs. 
A nesting on a modular graph is a possibly empty collection of nests such that each pair is either disjoint or nested (one is contained in the other).  
To every nest $N$ on a graph $\gamma$ we associate the stable graph $\hat{N}$ by adjoining to $N$ all adjacent half-edges in $\gamma$ as the legs of $\hat{N}$.

A weak modular operad is a stable $\mathbb{S}$-module along with an operation for each nested modular graph.  These operations are in turn specified by $S_n$-equivariant {\it generating} operations 
\begin{equation} \label{genops}
\mu_\gamma\colon A(\gamma)\tensor_{Aut(\gamma)}\Sigma^{-1}\mathfrak{K}(\gamma)^\ast \to A(g,n).
\end{equation} 
  Generating operations may be composed via graph insertion and the resulting operation is indexed by a nested modular graph.  These compositions are subject to a differential constraint $d(\mu_\gamma) = \sum_N \mu_{\gamma/N} \circ_N \mu_{\hat{N}}$. 
   We refer to \cite[Proposition 3.21]{WardMP} for more detail, and discussion of sign conventions for composing generating operations. There is a parallel notion of weak $\mathfrak{K}$-modular operads in which the generating operations are of the form
   \begin{equation}\label{weakk}
       \mu_\gamma\colon \Sigma^{-1} A(\gamma)_{Aut(\gamma)} \to A(g,n).  
   \end{equation}

Observe that a modular (resp.\ $\mathfrak{K}$-modular) operad is a weak modular (resp.\ weak $\mathfrak{K}$-modular) operad for which generating operations corresponding to graphs with more than one edge vanish.

\subsection{The Feynman transform.}
Following \cite{GeK2}, we momentarily restrict attention to modular operads which are finite dimensional in each graded component (see Remark $\ref{coFTrmk}$). The weak Feynman transform is a pair of functors
$\FT^+ \colon \left\{\text{weak modular operads} \right\} \to \left\{\mathfrak{K}\text{-modular operads} \right\}$
and
$\FT^- \colon \left\{\text{weak }\mathfrak{K}\text{-modular operads} \right\} \to \left\{\text{modular operads} \right\} $
defined as follows.  Given a weak modular operad $A$, define 
\begin{equation}\label{FT}
\FT^+(A)(g,n) = 
\ds\bigoplus_{\gamma \in \Gamma(g,n)}
 (\mathfrak{K}(\gamma)\tensor A^\ast(\gamma))_{Aut(\gamma)}
\end{equation}
along with the differential given by the sum of the linear dual of the generating operations in Equation $\ref{genops}$.  This dg stable $\mathbb{S}$-module carries the structure of a $\mathfrak{K}$-modular operad by graph insertion at vertices.

The functor $\FT^-$ is defined analogously, except we exclude the tensor factor $\mathfrak{K}(\gamma)$ in Equation $\ref{FT}$.  We refer to \cite[Section 4]{WardMP} for more details.  In what follows we write $\FT$ in place of $\FT^+$ and $\FT^-$, since the domain will be clear from the context.  When restricted to modular (resp.\ $\mathfrak{K}$-modular) operads, these functors recover the Feynman transform as defined in \cite{GeK2}.  We often use the notation $\FTGK$ in place of $\FT$ when the input is strict.

We will use the following properties of the (weak) Feynman transform.  The functors $\FT$ take $\infty$-morphisms to strict morphisms.  They take $\infty$-quasi-isomorphisms to quasi-isomorphisms. Finally, if $A$ is a modular operad then there exists a quasi-isomorphism $\FT^2(A)=\FTGK^2(A) \stackrel{\sim}\to A$.

\begin{remark}\label{coFTrmk}
Modular operads may be encoded as algebras over a particular (groupoid colored) operad, see \cite{WardMP}.  From this perspective, the Feynman transform is naturally isomorphic to the linear dual of the operadic bar construction for algebras over this operad (see Equation $\ref{naturalisos}$).  
   When performing combinatorial analysis of the differentials below, it will be preferable to have both of these dual constructions on hand.
\end{remark}

\subsection{$\infty$-morphisms of weak modular operads}
In this section we give an axiomatic description of the notion of an $\infty$-morphism between weak $\mathfrak{K}$-modular operads.  For this discussion we fix two such weak $\mathfrak{K}$-modular operads $\op{Q}$ and $\op{P}$; we will subsequently be interested in the case $\op{Q}=\DM$ and $\op{P}= \mathfrak{p}H_\ast(\op{M})$.  There is also a parallel notion of $\infty$-morphism between weak modular operads \cite{WardMP}, but it won't be needed in this paper.

An $\infty$-morphism of $\mathfrak{K}$-modular operads $f\colon \op{Q}\rightsquigarrow \op{P}$ specifies an $S_n$-equivariant map
\begin{equation}\label{im}
f_\gamma \colon \op{Q}(\gamma)_{Aut(\gamma)} \to \op{P}(g,n)
\end{equation}
of degree $0$ for each $\gamma \in  \Gamma(g,n)$.
These morphisms are subject to the following differential constraint.  Denote by $\nu_\gamma\colon \op{Q}(\gamma)_{Aut(\gamma)}\to \op{Q}(\ast_\gamma) $ and $\mu_\gamma\colon \op{P}(\gamma)_{Aut(\gamma)}\to \op{P}(\ast_\gamma)$ the generating operations for the weak $\mathfrak{K}$-modular operadic  on $\op{Q}$ and $\op{P}$ (Equation $\ref{weakk}$).
Recall these maps have degree $-1$.  A family of maps as in Equation $\ref{im}$ is an $\infty$-morphism if and only if it satisfies:
\begin{equation}\label{dcond0}
\sum_{N \text{ on } \gamma} (f_{\gamma/N}\circ_N \nu_N ) = \ds\sum_{V(\gamma) = \sqcup_i V(N_i)} (\mu_{\gamma/\cup N_i} \circ \tensor_i f_{N_i}).
\end{equation}

The first sum is taken over all nests $N$ on $\gamma$, along with the ``empty nest'' with the convention $\nu_\emptyset=d_\op{Q}$.  
The second sum is taken over nestings $\{N_i\}$ of $\gamma$ such that each vertex of $\gamma$ is contained within exactly one nest.  For future reference we call such a collection of nests a depth 1 nesting.  Collapsing the edges in each nest in such a nesting gives a modular graph $\gamma/\cup N_i$.  Again this includes a term of the form $d_\op{P}\circ f_\gamma$ under the convention $\mu_{\gamma/\gamma} = d_\op{P}$.  We remark that the signs in this expression are particularly simple because the input was already $\mathfrak{K}$-twisted.  In the non-$\mathfrak{K}$-twisted case we would have Koszul signs coming from rearranging tensor factors of the form $\mathfrak{K}(\gamma)$.

Below, we will specialize to the case where $\op{P}$ has no internal differential and $\op{Q}$ is a (strict) $\mathfrak{K}$-modular operad.  In this case the differential condition can be written as:
\begin{equation}\label{inftydcond}
\partial (f_\gamma) = f_\gamma\circ d_{\op{Q}}  = -\sum_{|N|=1} (f_{\gamma/N}\circ_N \nu_N ) + \ds\sum_{\substack{N_i\neq \gamma \\ V(\gamma) = \sqcup_i V( N_i)}} (\mu_{\gamma/\cup N_i} \circ \tensor f_{N_i})
\end{equation}
where $|N|=1$ means nests with one edge.

\subsection{Truncated modular operads.}\label{sec:truncated}

  A non-empty set $S\subset \{(g,n) : 2g+n\geq 3 \}$   of stable pairs is called closed if for each modular graph $\gamma$ of a type $(g,n)\in S$, and each nest $N$ on $\gamma$, the type of the stable graph $\hat{N}$ is in $S$.  In particular this means all vertices of $\gamma$ are of a type in $S$.

Let $S$ be a closed set of stable pairs.  Define an $S$-truncated modular operad (or an $S$-modular operad for short) to be a collection of dg vector spaces $A(g,n)$ for each $(g,n) \in S$ along with associative, $S_n$-equivariant operations $A(\gamma)_{Aut(\gamma)} \to A(g,n)$, for every modular graph $\gamma$ of type $(g,n)\in S$.  In particular every modular operad may be viewed as an $S$-modular operad by simply forgetting operations which don't land in $S$. 
  We similarly define $\mathfrak{K}$-twisted $S$-modular operads to be an $S$-indexed collection with such operations $A(\gamma) \tensor_{Aut(\gamma)} \mathfrak{K}(\gamma) \to A(g,n)$ for each modular graph $\gamma$ of type $(g,n)\in S$.

\begin{lemma}\label{triple}  Let $S\subset T$ be closed sets of stable pairs.  The forgetful functor from $T$-modular operads to $S$-modular operads has both a left and right adjoint.  
\end{lemma}

This follows from \cite[Proposition 7.9]{W6}.  Indeed the definition of closed is simply a translation of the requirements of that proposition.  
 For a given $S\subset T$ we denote this triple of adjoint functors $(\iota_\ast,\iota^\ast,\iota_!)$.  The functor $\iota^\ast$ forgets those operations not living in $S$, the functor $\iota_!$ is extension by $0$, and the functor $\iota_\ast$ is a left Kan extension.  In particular $\iota^\ast\circ\iota_!$ and $\iota^\ast\circ\iota_\ast$ are identity functors.

One choice of $T$ would be all stable pairs, in which case a $T$-modular operad is simply a modular operad.   We define the Feynman transform of an $S$-modular operad to be $\FT_S := \iota^\ast \circ \FT \circ \iota_!$.  Defined as such, it is straight forward to verify that the Feynman transform commutes with the forgetful functor.  
Thus information about the Feynman transform in biarity $(g,n)$ is completely contained in the restricted Feynman transform $\FT_S$ anytime $(g,n)\in S$.

\begin{example}  Here are several examples of closed sets of stable pairs.
	\begin{itemize}
		\item Ordering the set of stable pairs lexicographically, so that $(g,n) < (g,n+1) < (g+1, n)$, $L_{(h,m)} = \{(g,n)<(h,m)\}$ is closed.  A $L_{(1,0)}$-modular operad is a cyclic operad; a $L_{(2,0)}$-modular operad is a semi-classical modular operad.
		\item Fix an integer $a\geq 1$.  The set $T_a = \{(g,n) \  |  \ 2g+n -2 \leq a\}$ is closed.
		\item Fix a stable pair $(h,m)$.  The intersection of all closed sets containing $(h,m)$ is closed, call it $M_{(h,m)}$. Observe $M_{(h,m)} \subset T_{2h+m-2}$, but they needn't be equal. For example $(2,0) \not\in M_{(2,1)}$.
		\item Fix an integer $a\geq 0$.  The set $D_a = \{(g,n) \ | \ 3g+n - 3 \leq a\}$ is closed.
	\end{itemize}
\end{example}

This last family of examples was considered in \cite{GNPR} for which a tower of modular operads was constructed.  The following is a straight forward generalization of their construction.  Call a sequence of closed subsets $S_0\subset S_1\subset...$  exhaustive if each $(g,n)$ is contained in some $S_i$.  Given an exhaustive sequence $\{S_i\}$ of closed subsets, denote the triple of adjoint functors connecting $S_{i}$-modular and $S_{i+1}$-modular operads as $(\phi^i_\ast, \phi^\ast_i, \phi_!^i)$.  Fix a modular operad $\op{A}$.  Taking the adjoint form of $\phi_{i-1}^\ast$ applied to the unit $\op{A} \to \phi^\ast_i \phi^i_\ast(\op{A})$ we get a sequence of modular operads:
$$ 0 \to \dots \to \phi_\ast^{i-1} \phi_{i-1}^\ast\op{A} \to \phi_\ast^{i} \phi_{i}^\ast\op{A} \to \dots \to \op{A}.$$

This construction may be applied to each of the examples above.  In Section $\ref{minmod}$ we will be interested in the tower associated to $D_\bullet$ and its relation to minimal models for modular operads, after \cite{GNPR}.  In Section $\ref{dmsec}$ we will be interested in the tower associated to $L_\bullet$, with a particular focus at the stage $(g,n)=(1,4)$.  For this we define:

\begin{definition}\label{truncated} A $(g,n)$-truncated modular (resp.\ $\mathfrak{K}$-modular) operad is defined to be a $L_{(g,n+1)}$-modular (resp.\ $\mathfrak{K}$-modular) operad.
\end{definition}

In particular, the (1,4)-truncation of a modular operad regards the genus 0 portion along with biarities $(1,n)$ for $n\leq 4$.  From the perspective of constructing higher operations for larger and larger truncations, the towers associated to $T_\bullet$ and $M_\bullet$ listed above are potentially simpler.  In general one needn't construct higher operations landing in $(g,n+3)$ in order to understand operations landing in $(g+1,n)$.

\subsection{The weak Feynman transform is minimal.} \label{minmod}
Let $\op{P}$ be a modular operad and let $V$ be a stable $\mathbb{S}$-module with differential $d_V=0$.  Let $\xi \colon V \to \op{P}$ be an $\mathbb{S}$-module map of degree $-1$.  By definition \cite{GNPR} the principal extension of $\xi$, denoted $\op{P}\sqcup_\xi V$ is a modular operad representing the functor from modular operads to sets given by:
\begin{equation}
\op{Q}\mapsto
\left\{ (f,g) \in Hom_{\text{ModOp}}(\op{P},\op{Q}) \times Hom_{\text{SMod}}(V,\op{Q}) \ | \ d_\op{Q}\circ g=f\circ \xi \right\}.
\end{equation}
Note there is a natural map of modular operads $\iota\colon \op{P}\to \op{P}\sqcup_\xi V$. We say a map of modular operads $\phi\colon\op{P}_1\to\op{P}_2$ is a principal extension if
there exists a $\xi\colon V\to \op{P}_1$ for which there is an isomorphism $\psi\colon \op{P}_2\stackrel{\cong}\to \op{P}_1\sqcup_\xi V$ with $\iota = \psi\circ\phi$.

Consider the tower of modular operads
\begin{equation}\label{tower}
0 \to \dots \to \phi^{r-1}_\ast \phi_{r-1}^\ast\op{P} \to \phi^r_\ast \phi_r^\ast \op{P} \to \dots \to \op{P}
\end{equation}
which corresponds to the exhaustive sequence $D_r$.  So here $\phi^\ast_r$ restricts to modular dimension $3g-3 +n\leq r$ and $\phi_\ast^r$ is its left adjoint.  Notice that our notation for the left and right adjoints of the forgetful functor follows \cite{W6} and is opposite to \cite{GNPR}.

\begin{lemma}  Let $\op{A}$ be a $\mathfrak{K}$-twisted modular operad with internal differential $d_\op{A}=0$.  For each $r$, the natural map
		\begin{equation*}
	\phi^{r-1}_\ast \phi_{r-1}^\ast\FT(\op{A}) \to \phi^{r}_\ast \phi_{r}^\ast\FT(\op{A}) 
	\end{equation*}
	is a principal extension.  Explicitly  $\phi^{r}_\ast \phi_{r}^\ast\FT(\op{A}) \cong \phi^{r-1}_\ast \phi_{r-1}^\ast(\FT(\op{A}))\sqcup_{\xi} \op{A}^\ast_r$.
\end{lemma}
\begin{proof}
	The natural map $\pi_r\colon  \phi^{r-1}_\ast \phi_{r-1}^\ast\FT(\op{A}) \to \phi^{r}_\ast \phi_{r}^\ast\FT(\op{A})$ is described as follows.	The source and target are both quotients of the free modular operad on the modular operad $\FT(\op{A})$.  A homogeneous element in the free modular operad on $\FT(\op{A})$ is equivalent to a depth 1 nested graph $(\gamma, N_i)$ whose vertices $v$ are labeled by vectors in $\op{A}^\ast(v)$; nests are viewed as labeled by elements of $\FT(\op{A})$.  The quotient in the source identifies the composition of nests along connected subgraphs of $\gamma$ having modular dimension less than or equal to $r-1$.  	The quotient in the target makes the same identifications, but up to modular dimension $r$.  The map $\pi_r$ is the quotient map.

When a graph $\gamma$ is of modular dimension less than $r$, all possible nestings are identified in $\phi^{r}_\ast \phi_{r}^\ast\FT(\op{A})$, and so there is a natural isomorphism $(\phi^{r}_\ast \phi_{r}^\ast\FT(\op{A}))_{\leq r} \cong \FT(\op{A})_{\leq r}$.  The weak Feynman transform differential decreases modular dimension.  Consider the composition of this natural isomorphism with the Feynman transform differential restricted to $\op{A}^\ast_r$:
\begin{equation*}
\op{A}^\ast_r \stackrel{d_{\FT}}\longrightarrow  \FT(\op{A})_{\leq r-1}\cong (\phi^{r-1}_\ast \phi_{r-1}^\ast\FT(\op{A}))_{\leq r-1}.
\end{equation*}
Call this composition $\xi$.  It is a degree $-1$ map of stable $\mathbb{S}$-modules.

We then check the universal property.  Given a dg map $h\colon \phi_\ast^r\phi_{r}^\ast\FT(\op{A})\to \op{Q}$, we precompose with $\pi_r$ and $\iota\colon \op{A}^\ast_r\hookrightarrow  \phi_\ast^r\phi_{r}^\ast\FT(\op{A})$ to get the two maps $f$ and $g$.  If we precompose $dh=hd$ with $\iota$ and use the fact that $\pi\circ\xi = d_{\FT}\circ i$, 
we find  $d_\op{Q}\circ g = f \circ \xi$.

	Conversely, suppose that we have such a pair $(f,g)$ with  $d_\op{Q}\circ g = f \circ \xi$.
	To define $h\colon \phi_\ast^r\phi_{r}^\ast\FT(\op{A})\to \op{Q}$ it suffices, by left adjointness of $\phi_\ast^r$, to define the map up to and including modular dimension $r$.  It thus suffices to define a map $h\colon \FT(\op{A})_{\leq r} \to \op{Q}$.  
	Since the truncated modular operad structure on $\FT(\op{A})_{\leq r}$ is free, this map is determined by where it sends $\op{A}^\ast_{\leq r}$.  We use $f$, precomposed with inclusion, to define $h$ on $\op{A}^\ast_{\leq r-1}$.  We use $g$ to define $h$ on $\op{A}^\ast_{r}$.  
	
	It remains to check that the induced map is dg, and it suffices to check this on the image of the inclusion of the generators $\op{A}^\ast_{\leq r}\hookrightarrow \FT(\op{A})_{\leq r}$   Taking the differential of a homogeneous element of modular dimension less than $r$, this will follow from the fact that $f$ is dg.  If we take the differential of a homogeneous element of modular dimension equal to $r$, this will follow by combining $\pi\circ\xi = d_{\FT}\circ i$ with the assumption that $d_\op{Q} \circ g = f\circ\xi$.
\end{proof}

We recall \cite{GNPR} that a modular operad is called minimal if the tower in Equation $\ref{tower}$ is a sequence of principal extensions.  Thus we have shown:
\begin{corollary}  The weak Feynman transform of a $\mathfrak{K}$-modular operad with vanishing internal differential is a minimal modular operad.
\end{corollary}
We remark that combining this result with Theorem $\ref{thm3}$ establishes Theorem $\ref{thm4}$.

\section{Analysis of $\DM$}\label{dmsec}

In this section we consider the Feynman transform of the standard modular operad structure on the stable $\mathbb{S}$-module $H_\ast(\overline{\op{M}})(g,n):= H_\ast(\overline{\op{M}}_{g,n})$ after \cite{GeK2}.

Using Theorem $\ref{thm1}$ we now fix an $\mathbb{S}$-module isomorphism $H_\ast(\DM)\cong\mathfrak{p}H_\ast(\op{M})$ chosen to satisfy two additional properties.  The first is that the corolla labeled by the dual fundamental class in $H_{6g+2n-6}(\overline{\op{M}}_{g,n})^\ast$ maps to the homology class of a point in $\Sigma^{6-6g-2n}H_0(\op{M}_{g,n})$.   
Second we may furthermore impose that our isomorphism is chosen so that the maps in genus $0$ combine to form a map of cyclic operads with respect to the standard operad structure on $\mathfrak{p}H_\ast(\op{M}_{0,n})$, seen as a suspension of the gravity operad \cite{Getgrav}.  The fact that such an isomorphism exists follows from Koszul duality of the hyper-commutative and gravity operads \cite{GetHyCom}.  We endow $\mathfrak{p}H_\ast(\op{M})$ with the structure of a $\mathfrak{K}$-modular operad induced from this chosen isomorphism.

There is a map of $\mathfrak{K}$-modular operads
$\DM \to H_\ast(\DM)$ induced by sending a corolla labeled by a dual fundamental class to the homology class it represents.  Define $f^0$ to be the composite
\begin{equation} \label{zeta}
\DM \to H_\ast(\DM)\stackrel{\cong}\to\mathfrak{p}H_\ast(\op{M}).
\end{equation}

\subsection{Bigrading of $\DM$}

The complex $\FTGK(H_\ast(\overline{\op{M}}))(g,n)$ is bigraded.   A homogeneous element of bidegree $(-r,-s)$ is given by a labeled graph with $r$ edges and internal degree $-s$.  Explictily $-s$ is the sum of the degrees of the vertex labels and so $r,s \geq 0$. The differential in the Feynman transform preserves internal degree, so we have a splitting of complexes:
\begin{equation}\label{split}
\FTGK(H_\ast(\overline{\op{M}}))(g,n) = \ds\bigoplus^{6g+2n-6}_{s=0}\FTGK(H_\ast(\overline{\op{M}}))(g,n)^{\bullet,-s}
\end{equation}

We remark that this splitting induces a secondary grading on the homology $H_\ast(\op{M}_{g,n})$ recovering the associated graded of the weight filtration.  Specifically, for a given $r$ and $s$ this complex computes the weight $s$ part of the compactly supported (co)homology in degree $r+s$ (see e.g. \cite{AWZ}), which coincides with the weight $6g+2n-6-s$ part of the ordinary cohomology in degree $6g+2n-6-r-s$ by Poincare duality. 

\begin{lemma}\label{deglem} Each complex $\FTGK(H_\ast(\overline{\op{M}}))(g,n)^{\bullet,-s}$ is supported in external degree $0 \leq r \leq \frac{1}{2}(6g+2n-6-s)$.
\end{lemma} 
\begin{proof}  Fix $s$ and a non-zero homogeneous vector in bidegree $(-r,-s)$.  Let $\gamma$ be the underlying graph with genus $\beta_1(\gamma)$, with $E=r$ edges and with $V$ vertices indexed by $i$.  By definition $g= \beta_1(\gamma) + \sum_i g_i$, where $(g_i,n_i)$ is the type of each vertex.  By dimension considerations 
	\begin{equation*}
	s \leq \sum_i (6g_i + 2n_i -6) = 6g - 6\beta_1(\gamma) -6V + \sum_{i}2n_i.	
	\end{equation*}
  Counting flags we have $2E + n = \sum n_i$, and thus 
$	s \leq   6g - 6\beta_1(\gamma) -6V + 4E + 2n.	$
Calculating Euler characteristic of $\gamma$ we have $V-E = 1 - \beta_1(\gamma)$, and thus
$s \leq   6g - 2E-6 + 2n$.
\end{proof}

When $s$ is even, we see this upper bound is achieved when all vertices are labeled by dual fundamental classes.

\begin{lemma}\label{row2lem} Let $g\geq 1$ and $(g,n)\neq (1,1)$.  A homogeneous element in $\FTGK(H_\ast(\overline{\op{M}}))(g,n)^{4-3g-n,-2}$ is not a $d_{\FTGK}$ boundary only if its underlying graph has no loops, no parallel edges and no vertices of genus $\geq 1$.
\end{lemma}
\begin{proof}  First observe that the Feynman transform differential restricted to $H_2(\overline{\op{M}}_{1,2})^\ast$ is injective, since $H_2(\op{M}_{1,2})=0$ by \cite{Harer}.
	Thus, the $S_2$-representation $H_2(\overline{\op{M}}_{1,2})\cong V_2\oplus V_2$ has a basis given by the image of a loop contraction $H_2(\overline{\op{M}}_{0,4})\to H_2(\overline{\op{M}}_{1,2})$ and a non-loop contraction $H_2(\overline{\op{M}}_{1,1})\tensor H_0(\overline{\op{M}}_{0,3})\to H_2(\overline{\op{M}}_{1,2})$.  Let $u^\ast$ and $v^\ast$ denote the dual basis of $H_2(\overline{\op{M}}_{1,2})^\ast$.

	Consider a non-zero homogeneous vector in $\FTGK(H_\ast(\overline{\op{M}}))(g,n)^{4-3g-n,-2}$ supported on a graph $\gamma$.  By Lemma $\ref{deglem}$ and its proof we see that $\gamma$ must be labeled by dual fundamental classes.  It follows that $\gamma$ has exactly one vertex not of type (0,3), which must be of of type (0,4) or (1,1).  	If the graph $\gamma$ contains parallel edges (a circuit with 2 edges), the vector itself is seen to be zero after modding out by the automorphism of $\gamma$ which permutes these parallel edges.  We thus assume $\gamma$ has no parallel edges. 

	We now prove the Lemma by constructing preimages of $d_\FTGK$ in the remaining cases.  	
	If the distinguished vertex if of type (1,1), our assumption that 
	$(g,n)\neq (1,1)$ implies this vertex is adjacent to one edge.  Contract this edge to get a graph whose only vertex not of type $(0,3)$ is labeled by a non-zero scalar multiple of $v^\ast$, for which the image under $d_\FTGK$ recovers the original vector.  
	Now assume the distinguished vertex is of type $(0,4)$.
	If the graph $\gamma$ contains a loop (a circuit with one edge) we may contract it to a vertex which has either valence 1 or 2.  In each case we may label this vertex so that the unique term in the Feynman tranform differential is the vector we started with.  In the latter case this label will be a non-zero scalar multiple of $u^\ast$.
	
	To conclude the proof we have only to verify that the constructed preimages of $d_{\FTGK}$ do not vanish upon passage to coinvariants by graph automorphisms.  The fact that no parallel edges were created and that each non-distinguished vertex is of type $(0,3)$ ensures this.
\end{proof}

This lemma can be rephrased as saying that, for $(g,n)>(1,1)$,  a non-trivial cycle of minimal degree on the $s=2$ row of $\DM(g,n)$ must have a unique vertex of type $(0,4)$, all remaining vertices of type $(0,3)$ and only polygonal circuits.  These non-trivial cycles represent (co)homology classes in $\mathfrak{gr}^W_{6g+2n-8} H^{3g+n-4}(\op{M}_{g,n})$, and we will use this result to find representing cycles for   $\mathfrak{gr}^W_{6} H^{3}(\op{M}_{1,4})$ below (Lemma $\ref{hlem}$).

\subsection{Analysis of the double complex $\DM(g,n)$ for $(g,n) < (1,4)$}
\begin{lemma}\label{13lem}  Let $(g,n)<(1,4)$.  If $2r+s \neq 6g+2n-6$ then $H_\ast(\DM(g,n))^{-r,-s}=0$.
\end{lemma}
\begin{proof}   
For $g=0$ this follows from \cite{GetHyCom}.  For $g=1$ we remark that one can use the calculation of the $S_n$ equivariant Poincare/Hodge polynomials of $\overline{\op{M}}_{g,n}$ from \cite{GetHyCom} and \cite{GetSC} to easily count dimensions of each bigraded component and hence find the Euler characteristic of each row.  Then for (1,1) and (1,2), the result follows from the fact \cite{Harer} that $H_1(\op{M}_{1,n})= H_2(\op{M}_{1,n})=0$.  Finally for the case (1,3), we may use \cite[Theorem 1.2]{CGP} to determine the $s=0$ row, \cite[Theorem 1.1]{PetG1} for the $s=2$ row and \cite{Harer} for the $s=4$ row.  
\end{proof}

\begin{corollary}\label{13formal}  The (1,3)-truncation of $\DM$ is formal.
\end{corollary}
\begin{proof}  
	For each $(g,n)\leq (1,3)$, the corolla labeled by the dual fundamental class determines a non-boundary cycle in $\DM$.  Consider the unique $\mathfrak{K}$-modular operad map $\DM \to H_\ast(\DM)$ induced by taking each such labeled corolla to its homology class.  		
	Combining Lemma $\ref{13lem}$ and the proof of Lemma $\ref{deglem}$, one sees that in this range homology classes in $\DM$ are supported on graphs whose vertices are labeled by dual fundamental classes. Thus the (1,3)-truncation of the induced map on homology will be surjective, and hence an isomorphism.
\end{proof}

With respect to the $\mathfrak{K}$-modular operad map in Equation $\ref{zeta}$, this result can be restated as:

\begin{corollary} \label{13qi} 	The (1,3)-truncation of the morphism $f^0$ is a quasi-isomorphism.
\end{corollary}

\subsection{Analysis of the double complex $\DM(1,4)$}
To present the obstruction to formality in Theorem $\ref{thm2}$, we look at this double complex in biarity $(1,4)$.  This obstruction lies in internal degree $s=4$, although it will be helpful for us to calculate the entire homology as a graded $S_4$-module.  The heart of this calculation is the fact that $\mathfrak{gr}^W_4 H^4(\op{M}_{1,4})=0$, \cite[Theorem 1.1]{PetG1}.

\begin{lemma} \label{hlem}  
	\begin{equation*}
 H_\ast(\FTGK(H_\ast(\overline{\op{M}}))(1,4))^{-r,-s} \cong \begin{cases}
	V_{3,1} & \text{ if } s=0 \text{ and } r=4 \\ 
V_{2,1,1} & \text{ if } s=2 \text{ and } r=3 \\
	V_{4} & \text{ if } s=4 \text{ and } r=1\\
	V_{4} & \text{ if } s=8 \text{ and } r=0 \\
	0 & \text{ else } 
\end{cases}
\end{equation*}	
\end{lemma}
\begin{proof}  The case $s=8$ is trivial.
	Adapting the results used in the proof of Lemma $\ref{13lem}$ takes care of the cases $s=0,6$.  For the case $s=4$ we again use \cite[Theorem 1.1]{PetG1}, but here combine it with \cite{Harer} to determine that the only homology can occur when $r=1$.  We then compute that the Euler characteristic of the $s=4$ row is $23-60+36=-1$, which equals the Euler characteristic of the $S_4$ invariant subcomplex of $\FTGK(H_\ast(\overline{\op{M}}))(1,4)^{\bullet,-4}$, which is $7-14+6 = -1$ (see Lemma $\ref{Qcomplex}$ below) to complete the calculation in the case $s=4$.

It remains to consider the case $s=2$.  Denote the differential emanating from bidegree $(-r,-s)$ as $d^{-r,-s}$.
	We first compute $coker(d^{-2,-2}) \cong V_{2,1,1}$ as follows.
	Looking at bidegree  $(-3,-2)$, consider the subspace of triangles with 2 legs on one vertex and one leg on the other two vertices, whose vertices are labeled by dual fundamental classes.  This is a $6$-dimensional space of cycles, whose $S_4$-representation type is $V_{3,1}\oplus V_{2,1,1}$.  This $6$ dimensional space can only be hit by the differential coming from $(1,4)$ graphs with parallel edges
	whose two vertices are labeled by classes of internal degrees $0$ and $-2$ respectively.  Taking coinvariants with respect to the graph automorphism which permutes these parallel edges, and using the respresentation type of $H_2(\overline{\op{M}}_{0,5})$ via \cite[Theorem 5.9]{GetHyCom}  we see this space is isomorphic to the standard representation $V_{3,1}\oplus V_{4}$.  Thus $coker(d^{-2,-2})$ must include this copy of $V_{2,1,1}$.

To show that this is the entire cokernel, we invoke Lemma $\ref{row2lem}$ to conclude the remainder of the bidegree $(-3,-2)$ part is in the image of the Feynman transform differential.
It remains to show that  $d^{-2,-2}$ restricted to the $V_{3,1}$ summand of the four dimensional space corresponding to (1,4) graphs with parallel edges has full rank.  This differential is induced by applying the Feynman transform differential to the unique (45) alternating class in $H_2(\overline{\op{M}}_{0,5})^\ast$, and so this is a straight forward calculation using the WDVV relations.

Having computed $coker(d^{-2,-2}) \cong V_{2,1,1}$ we observe that (by Lemma $\ref{deglem}$) this complex is supported on total degree $\geq -5$, so $coker(d^{-2,-2})$ is the homology in bidegree $(-3,-2)$.  Using the fact $H_{k\geq 5}(\op{M}_{1,4})=0$ \cite[Theorem 4.1]{Harer}, 
the only other homology could lie in bidegree $(-2,-2)$, but a quick counting of graphs shows the Euler characteristic of the $s=2$ row to be $12-60+91-46 =-3$, so no such homology exists.
\end{proof}

\begin{remark}  We could circumvent the use of weights in the above proof if we knew {\it a priori} that the betti number $\beta_4(\op{M}_{1,4})=3$ (which is a consequence of the Lemma).  
\end{remark}

A consequence of the non-vanishing of $H_{-5}(\DM(1,4))^{-1,-4} \cong  \Sigma^{-8}\mathfrak{gr}^W_4 H^3(\op{M}_{1,4})$ is the non-formality of $\mathfrak{p}H_\ast(\op{M})$, as originally proven in \cite[Proposition 1.11]{AP}.

\begin{corollary} \label{notformal}
   There does not exist a (1,4)-truncated $\mathfrak{K}$-modular operad structure on $\mathfrak{p}H_\ast(\op{M})$ for which $\FTGK(H_\ast(\overline{\op{M}}))$ and $\mathfrak{p}H_\ast(\op{M})$ are quasi-isomorphic. 
\end{corollary}
\begin{proof}  	
	If $\DM$ and $\mathfrak{p}H_\ast(\op{M})$ were quasi-isomorphic, there would be a quasi-isomorphism $\DM\stackrel{\sim}\to \mathfrak{p}H_\ast(\op{M})$ (this follows from cofibrancy of the Feynman transform in the standard model structure for dg modular operads \cite[Theorem 8.4.2]{KW}).  This putative map would be defined by what it did to labeled corollas.  The necessity that cycles map to cycles and Lemma $\ref{hlem}$ combine to show that such a map can only map those corollas labeled by dual fundamental classes to something non-zero.  But since $\DM(1,4)^{-1,-4}$ has non-zero homology, the induced map on homology can't be surjective.
	\end{proof}

\begin{definition}\label{omegadef} We now fix, once and for all, a generator of $V_4\cong H_{-5}(\DM)(1,4)^{-1,-4}$.  We define $\omega \in \Sigma^{-8}H_3(\op{M}_{1,4})$ to be the image of this class under the isomorphism of Equation $\ref{zeta}$.
\end{definition}

\subsection{Analysis of the weight $s=4$, $S_4$-invariant part.}
By Corollary $\ref{notformal}$ the first obstruction to formality of $\DM$ lives in the $S_4$ invariant subcomplex of $\DM(1,4)^{\bullet, -4}$.  We may equivalently view this obstruction as living in the quotient complex of $S_4$-coinvariants.  In this section we will compute the differentials in this complex.

We start by considering the $S_4$-invariant subcomplex of the coFeynman transform (Remark $\ref{coFTrmk}$) in internal degree 4.  Explicitly that is the chain complex:
\begin{equation*}
(\ds\bigoplus_{\gamma \in \Gamma(1,4)_2}
H_4(\overline{\op{M}})(\gamma)\tensor_{Aut(\gamma)}\mathfrak{K}^{-1}(\gamma))^{S_4} \to 
(\ds\bigoplus_{\gamma \in \Gamma(1,4)_1}
H_4(\overline{\op{M}})(\gamma)\tensor_{Aut(\gamma)}\mathfrak{K}^{-1}(\gamma))^{S_4}  \to  H_4(\overline{\op{M}}_{1,4})^{S_4}
\end{equation*}
where $H_4(\overline{\op{M}})(\gamma)$ indicates that the sum of the degrees of the vertex labels of $\gamma$ is $4$, and $\mathfrak{K}^{-1}(\gamma)$ is the top exterior power of the set of edges, concentrated in degree $+|$edges of $\gamma|$.  The differential is given by the sum of edge contractions, under the convention that an edge $e$ in the last position in a wedge product $\dots\wedge e \in \mathfrak{K}^{-1}(\gamma)$ is removed upon contraction to produce a wedge product in $\mathfrak{K}^{-1}(\gamma/e)$. 
For notational simplicity we refer to this complex as $\op{Q} = \op{Q}_6\to\op{Q}_5\to\op{Q}_4$ from now on.

Figure $\ref{fig:14graphs}$ lists the elements in $\Gamma(1,4)^{S_4}_2$.   We have also chosen an auxiliary order on the edges of each graph for future reference.  We follow the notation given in \cite{GetGW}.  We use the notation $\delta_\bullet$ for an unspecified graph from the list.

\begin{figure}
	\centering
	\includegraphics[width=0.7\linewidth]{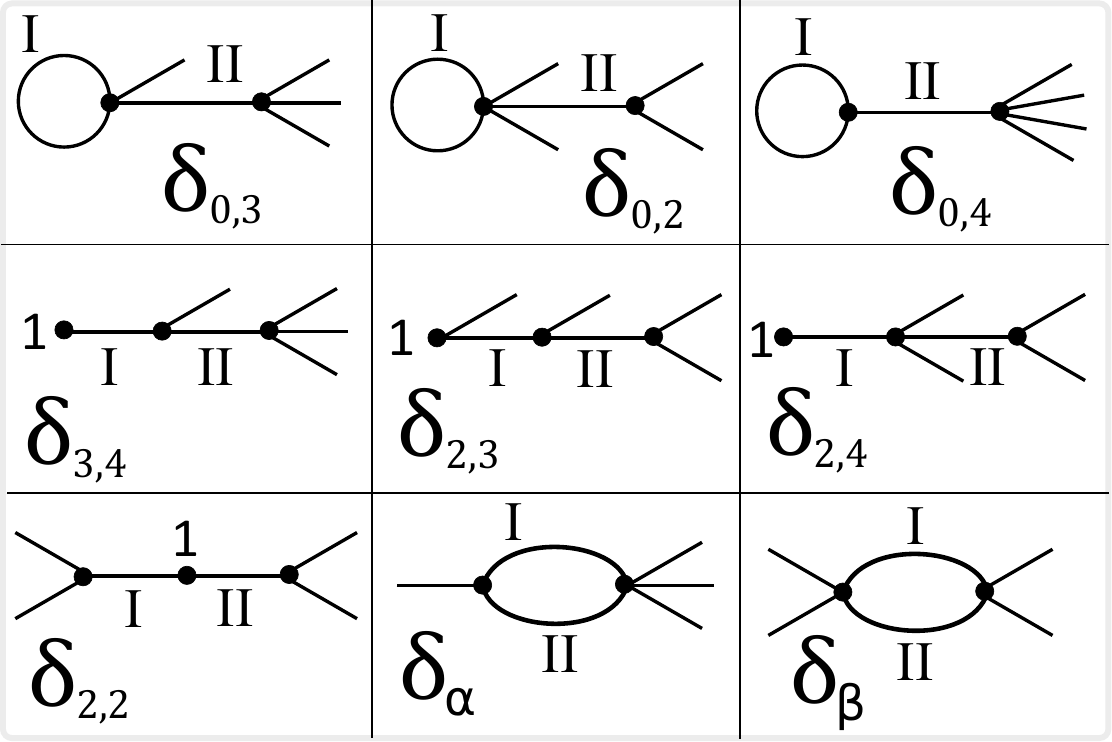}
	\caption{Stable $(1,4)$ graphs with two edges.  The edges are labeled I and II.  The genus of a vertex is indicated only for vertices of genus $1$.}
	\label{fig:14graphs}
\end{figure}

\begin{construction}  To each graph $\delta_\bullet$ we associate
\begin{itemize}
	\item An element in $\op{Q}_4$, which we also call $\delta_\bullet$.
			\item A pair of elements in $\op{Q}_5$, which we call $\delta^\text{I}_\bullet$ and $\delta^{\text{II}}_\bullet$.
	\item An element in $\op{Q}_6$, call it $\eta_\bullet$. 
\end{itemize}	
The element $\delta_\bullet$ is defined by labeling the vertices of the graph $\delta_\bullet$ with fundamental classes, summing over all leg labels and then composing each term using the modular operad structure.  Here, our conventions follow \cite{GetGW}: our convention is to use the topological fundamental class; to recover the stack/scheme theoretic fundamental class one should divide by the order of the stabilizer of a generic point. When summing over leg labels we disregard the auxiliary labeling of the edges.  So $\delta_{0,4}$ corresponds to one leaf labeled graph; $\delta_{0,2}$ corresponds to a sum of six leaf labeled graphs; $\delta_{2,2}$ corresponds to a sum of three labeled graphs and so on.

The element $\delta^\text{I}_\bullet$ is given by labeling the graph $\delta_\bullet$ with fundamental classes, summing over leg labels and then composing at the edge labeled by $\text{I}$ in each term.  Similarly the element $\delta^{\text{II}}_\bullet$ is given by contracting the edge labeled by $\text{II}$ in each term instead.  We emphasize that here, unlike in the previous paragraph, when we sum over leg labels we do not disregard the auxiliary edge labeling; if we did the result would not be $S_4$ invariant.  This distinction is only relevant in the graph $\delta_{2,2}$.  In particular both $\delta_{2,2}^\text{I}$ and $\delta_{2,2}^\text{II}$ are sums of (the same) six terms.

Finally the element $\eta_\bullet$ is given by labeling the graph $\delta_\bullet$ with fundamental classes, summing over leg labels and tensoring with $\text{I}\wedge \text{II} \in  \mathfrak{K}^{-1}(\delta_\bullet)$.  Again, we do not disregard the edge order when summing over leg labels, so $\eta_{2,2}$ is {\it a priori} a sum of six terms, although these terms cancel in pairs and $\eta_{2,2}=0$ as the following lemma indicates.
\end{construction}

\begin{lemma}\label{Qcomplex} The elements constructed above span $\op{Q}$, and a complete set of relations in each degree is given by:
\begin{enumerate}
	\item In $\op{Q}_6$ the $9$ vectors $\eta_\bullet$ satisfy $\eta_{2,2}=\eta_\alpha=\eta_\beta=0$.  In particular $\op{Q}_6$ is $6$-dimensional.
	\item In $\op{Q}_5$ the $18$ vectors $\delta_{\bullet}^\text{I}, \delta_\bullet^\text{II}$ satisfy 
	\begin{enumerate}
		\item $\delta_{2,2}^\text{I} = \delta_{2,2}^\text{II}$,  $\delta_{\alpha}^\text{I} = \delta_{\alpha}^\text{II}$, and  $\delta_{\beta}^\text{I} = \delta_{\beta}^\text{II}$,
	\item $	\delta_{0,2}^{\text{II}} + 3\delta_{0,3}^{\text{II}}+6 \delta_{0,4}^{\text{II}} = 3\delta_{\alpha}^{\text{II}} + 4\delta_{\beta}^{\text{II}}$.
		\end{enumerate}		
  In particular $\op{Q}_5$ is $14$-dimensional.
	\item In $\op{Q}_4$, the $9$ vectors $\delta_\bullet$ satisfy
		\begin{enumerate}
\item	$\delta_{0,2}+ 3\delta_{0,3}+6 \delta_{0,4} = 3\delta_{\alpha} + 4\delta_{\beta}, $
\item $12\delta_{2,2}-4\delta_{2,3} - 2 \delta_{2,4} + 6\delta_{3,4}  + \delta_{0,3} + \delta_{0,4}-2\delta_{\beta}=0.$
	\end{enumerate}
In particular $\op{Q}_4$ is $7$ dimensional.
\end{enumerate}
\end{lemma}
\begin{proof}
Since $\op{Q}_4 = H_4(\overline{\op{M}}_{1,4})^{S_4}$, the statements about $\op{Q}_4$ are \cite[Lemma 1.1]{GetGW} and \cite[Theorem 1.8]{GetGW}.  We next address the statements about $\op{Q}_6$.  Regarding $\eta_\alpha$ and $\eta_\beta$, passing to $Aut(\gamma)$ coinviants equates each term appearing in the sum with its negative.    Regarding $\eta_{2,2}$, we have a sum of six non-zero terms which cancel in pairs, seen by transposing $\text{I} \wedge \text{II}$.

It remains to consider the statements about $\op{Q}_5$.  Statement (a) is immediate, so these 18 vectors span a space of dimension $\leq 15$.  On other hand, the vector space $\op{Q}_5$ has a direct sum decomposition indexed by stable (1,4) graphs with one edge, of which there are four; see Figure $\ref{fig:141graphs}$.  Counting the dimensions of each summand, one sees that $\op{Q}_5$ is $3+3+4+4=14$ dimensional.  We show that a choice of 15 of the above 18 vectors not already known to be linearly dependent will span, and hence are subject to one additional relation.

\begin{figure}
	\centering
	\includegraphics[width=0.7\linewidth]{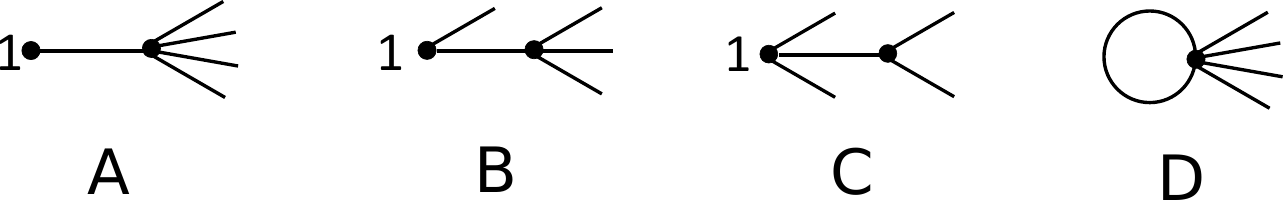}
	\caption{Stable $(1,4)$-graphs with one edge.}
	\label{fig:141graphs}
\end{figure}

Summand $\mathsf{A}$:  The vectors $\delta_{3,4}^\text{II}, \delta_{2,4}^\text{II}$ and $\delta_{0,4}^\text{I}$,  belong to this summand.  Comparing the non-fundamental class vertex labels; the first two are subject only to the WDVV relations \cite{GetHyCom} in $H_2(\overline{\op{M}}_{0,5})^{S_4}$, from which one checks their linearly independence.  The third lives in $H_0(\overline{\op{M}}_{1,1})$, hence these three classes are linearly independent and so span this three dimensional summand.

Summand $\mathsf{B}$: The vectors $ \delta_{3,4}^\text{I}$, $\delta_{0,3}^\text{I}$ and $\delta_{2,3}^\text{II}$,  belong to this summand.  
Comparing the non-fundamental class labels, the first two span the two dimensional space $H_2(\overline{\op{M}}_{1,2})$ (as in the proof of Lemma $\ref{row2lem}$), and the third lives in $H_0(\overline{\op{M}}_{0,4})$, so these three classes are linearly independent and hence span this three dimensional summand.

Summand $\mathsf{C}$: The vectors $ \delta_{0,2}^\text{I}$ $ \delta_{2,2}^\text{I}$, $\delta_{2,3}^\text{I}$ and $\delta_{2,4}^\text{I}$ belong to this summand.  Comparing the non-fundamental class vertex labels we have four classes in the four dimensional space $H_4(\overline{\op{M}}_{1,3})^{S_2}$.  If these classes did not span,  the Feynman transform differential emanating from  $H_4(\overline{\op{M}}_{1,3})^\ast$ would not be injective, which would violate Lemma $\ref{13lem}$.

Summand $\mathsf{D}$: The five vectors $ \delta_{0,3}^\text{II}$ $ \delta_{0,2}^\text{II}$, $\delta_{0,4}^\text{II}$ $ \delta_{\alpha}^\text{II}$ and $\delta_{\beta}^\text{II}$  belong to this summand.   Each has a unique vertex labeled by a class in the four dimensional space $(H_4(\overline{\op{M}}_{0,6})_{S_{\{5,6\}}})^{S_4}$.  These compositions are subject to the WDVV relation  from which we determine the relation and the linear independence of any subset of four of the five. 
\end{proof}
Using this lemma to choose bases, we now make explicit the differentials in the complex $\op{Q}$.  The differential $\op{Q}_6\to\op{Q}_5$ is given by:

\begin{center}
	$
	\begin{array}{r||c c c|c c c|c c c c|c c c c}
	\op{Q}_6\to\op{Q}_5 & \delta_{3,4}^\text{II} & \delta_{2,4}^\text{II} & \delta_{0,4}^\text{I} & \delta_{0,3}^\text{I} &  \delta_{3,4}^\text{I} & \delta_{2,3}^\text{II} & \delta_{2,4}^\text{I} &  \delta_{0,2}^\text{I} & \delta_{2,3}^\text{I} & \delta_{2,2}^\text{I} & \delta_{0,4}^\text{II} & \delta_{0,3}^\text{II} & \delta_{0,2}^\text{II} & \delta_{\beta}^\text{II} \\ 
	\hline 	\eta_{2,3} & 0 & 0 & 0 & 0 & 0 & 1 & 0  & 0 & -1 & 0 & 0 & 0 & 0 & 0 \\
	\hline 	\eta_{3,4} & 1 & 0  & 0 & 0 & -1 & 0 & 0 & 0 & 0 & 0 & 0 & 0 & 0 & 0 \\
	\hline 	\eta_{2,4} & 0 & 1  &  0 & 0 & 0 & 0 & -1 & 0 & 0 & 0 & 0 & 0 & 0 & 0 \\
	\hline 	\eta_{0,4} & 0 & 0 & -1 & 0 & 0 & 0  & 0 & 0 & 0 & 0 & 1  & 0 & 0 & 0 \\
	\hline 	\eta_{0,3} & 0  & 0 & 0 & -1 & 0 & 0 & 0 & 0 & 0 & 0 & 0 & 1 & 0 & 0  \\
	\hline 	\eta_{0,2} & 0 & 0 & 0 & 0 & 0 & 0 & 0 & -1 & 0 & 0 & 0 & 0 & 1 & 0	
	\end{array}
	$
\end{center}
The differential $\op{Q}_5\to\op{Q}_4$ is given by the transpose of:

\begin{center}
	$
	\begin{array}{r||c c c|c c c|c c c c|c c c c}
	(\op{Q}_5\to\op{Q}_4)^\text{T} & \delta_{3,4}^\text{II} & \delta_{2,4}^\text{II} & \delta_{0,4}^\text{I} & \delta_{0,3}^\text{I} &  \delta_{3,4}^\text{I} & \delta_{2,3}^\text{II} & \delta_{2,4}^\text{I} &  \delta_{0,2}^\text{I} & \delta_{2,3}^\text{I} & \delta_{2,2}^\text{I} & \delta_{0,4}^\text{II} & \delta_{0,3}^\text{II} & \delta_{0,2}^\text{II} & \delta_{\beta}^\text{II} \\ 
	\hline 	\delta_{2,2} & 0 & 0 & 0 & 0 & 0 & 0 & 0  & 0 & 0 & 2 & 0 & 0 & 0 & 6 \\
	\hline 	\delta_{2,3} & 0 & 0 & 0 & 0 & 0 & 1 & 0  & 0 & 1 & 0 & 0 & 0 & 0 & -2 \\
	\hline 	\delta_{3,4} & 1 & 0  & 0 & 0 & 1 & 0 & 0 & 0 & 0 & 0 & 0 & 0 & 0 & 3 \\
	\hline 	\delta_{2,4} & 0 & 1 & 0 & 0 & 0 & 0 & 1 & 0 & 0 & 0 & 0 & 0 & 0 & -1 \\
	\hline 	\delta_{0,4} & 0 & 0 & 1 & 0 & 0 & 0  & 0 & 0 & 0 & 0 & 1 & 0 & 0 & 1/2 \\
	\hline 	\delta_{0,3} & 0  & 0 & 0 & 1 & 0 & 0 & 0 & 0 & 0 & 0 & 0 & 1 & 0 & 1/2  \\
	\hline 	\delta_{0,2} & 0 & 0 & 0 & 0 & 0 & 0 & 0 & 1 & 0 & 0 & 0 & 0 & 1 & 0	
	\end{array}
	$
\end{center}

The linear dual of $\op{Q}$ is naturally isomorphic to the complex $\left[ \DM(1,4)^{\bullet,-4}\right]_{S_4}$.  If we abbreviate $A(\gamma):= H_4(\overline{\op{M}})(\gamma)\tensor \mathfrak{K}^{-1}(\gamma)$,
this natural isomorphism is given in each degree by:
\begin{equation}\label{naturalisos}
 (( \oplus_\gamma A(\gamma)_{Aut(\gamma)})^{S_4})^\ast 
 \to
  (( \oplus_\gamma A(\gamma)_{Aut(\gamma)})^\ast)_{S_4} 
  \to
  ( \oplus_\gamma (A(\gamma)^\ast)_{Aut(\gamma)})_{S_4}.
\end{equation}
The first map is given by precomposing a linear functional with the map $a \mapsto |S_4|^{-1}\sum_{\sigma \in S_4} \sigma a$ and then taking $S_4$ coinvariants of the result.  The second map is induced by precomposing a linear functional with the projection to $Aut(\gamma)$ coinvariants, and then taking $Aut(\gamma)$ coinvariants of the result.

Fixing the basis of $\op{Q}^\ast$ dual to the basis given in the above matrices, and composing with these natural isomorphisms gives a basis for $\left[ \DM(1,4)^{\bullet,-4}\right]_{S_4}$ for which the Feynman transform differential is given by the transpose of the matrices above.  In the middle dimension, it will be convenient to adopt the following notation:

\begin{definition}\label{lambdadef}  For $1\leq i \leq 14$, let $\lambda_i\in \left[ \DM(1,4)^{-1,-4}\right]_{S_4}$ be the image of the dual vector of the heading of the $i^{th}$ column in the above matrices under the natural isomorphisms in Equation $\ref{naturalisos}$.
\end{definition}

\subsection{A $(1,4)$-level-wise quasi-isomorphism.}

We now extend the level-wise quasi-isomorphism $f^0$ of Corollary $\ref{13qi}$ from the $(1,3)$-truncation to the $(1,4)$-truncation.  The result will not be a quasi-isomorphism of $\mathfrak{K}$ modular operads (lest Corollary $\ref{notformal}$ be violated), but will be the first map in an $\infty$-quasi-isomorphism of such.  For this construction we now fix an auxiliary scalar $\mathsf{e}\in\mathbb{Q}\setminus 0$.

Define $f_{(1,4)}$ as a sum of two maps:
\begin{equation*}
f_{(1,4)}=f^0_{(1,4)}+f^\omega \colon \FTGK(H_\ast(\overline{\op{M}}))(1,4) \to\mathfrak{p}H_\ast(\op{M}_{1,4}),  
\end{equation*} 
The first is the biarity $(1,4)$ component of the modular operad map $f^0$ (Equation $\ref{zeta}$) induced by sending corollas labeled by dual fundamental classes to point classes in $H_0$.  The second map $f^\omega$ is defined to be zero except on the bidegree $(-1,-4)$ component of this complex.  Using the basis $\{\lambda_i \ | \ 1\leq i \leq 14\}$ from Definition $\ref{lambdadef}$ we define:
\begin{equation*}
\bar{f}^\omega\colon\left[\FTGK(H_\ast(\overline{\op{M}}))(1,4)^{-1,-4}\right]_{S_4} \to \Sigma^{-8}H_3(\op{M}_{1,4}) 
\end{equation*}
by $\lambda_i \mapsto t_i\omega$, where $t_i$ are recorded in the following table:
\begin{equation}\label{t}
	\begin{tabular}{c|c|c|c|c|c|c|c|c|c|c|c|c|c}
		$t_1$ & $t_2$ & $t_3$ & $t_4$ & $t_5$ & $t_6$ & $t_7$ & $t_8$ & $t_9$ & $t_{10}$ & $t_{11}$ & $t_{12}$ & $t_{13}$ & $t_{14}$ \\ \hline 
		$-3\mathsf{e}$ & $\mathsf{e}$ & $-\mathsf{e}/2$ & $-\mathsf{e}/2$   & $-3\mathsf{e}$ & $2\mathsf{e}$  & $\mathsf{e}$ & $0$ & $2\mathsf{e}$  & $-6\mathsf{e}$ & $-\mathsf{e}/2$  & $-\mathsf{e}/2$   & $0$ & $2\mathsf{e}$
	\end{tabular}
\end{equation}
We then define $f^\omega$ to be $\bar{f}^\omega$ precomposed with projection to coinvariants.  Since $\omega$ is $S_4$-invariant, $f^\omega$ is $S_4$-equivariant.

\begin{lemma}\label{normlem}
	Let $\epsilon =\sum_{i=1}^{14}r_i\lambda_i \in \left[\FTGK(H_\ast(\overline{\op{M}}))(1,4)^{-1,-4}\right]_{S_4}$ and define
	\begin{equation*}
	|\epsilon| := (-6r_1 +2r_2-r_3-r_4+4r_6-6r_{10}+2r_{14})\mathsf{e}.
	\end{equation*}
	\begin{enumerate}
		\item If $\epsilon$ is a cycle then $\bar{f}^\omega(\epsilon)= |\epsilon|\omega$,
		\item $\epsilon$ is a boundary if and only if $\epsilon$ is a cycle and $|\epsilon|=0$.
	\end{enumerate}
\end{lemma}
\begin{proof} 
	Using the matrix representations of the Feynman transform differential above, this is easily verified.  First we see that such an $\epsilon$ is a cycle if and only if $r_1=r_5$, $r_2=r_7$, $r_3=r_{11}$, $r_4 = r_{12}$, $r_6=r_9$ and $r_8=r_{13}$, from which the first claim follows. For the second claim, if $\epsilon$ is a boundary it can be written as a linear combination of the seven vectors $d_{\FTGK}(\delta_\bullet^\ast)$, each of which satisfies $|d_{\FTGK}(\delta_\bullet^\ast)|=0$.  Conversely, if $\epsilon = \sum r_i\lambda_i$ is a cycle and $|\epsilon|=0$ then
$
	\epsilon = d_{\FTGK}(r_1 \delta_{3,4}^\ast + r_2\delta_{2,4}^\ast + r_3\delta_{0,4}^\ast + r_4\delta_{0,3}^\ast + r_6\delta_{2,3}^\ast +r_8\delta_{0,2}^\ast + \frac{r_{10}}{2}\delta_{2,2}^\ast).
$
\end{proof}

\begin{corollary}\label{dgcor} $f_{(1,4)}$ so defined is a dg map.
\end{corollary}
\begin{proof}  Since there is no differential in the target, this amounts to checking that $f_{(1,4)}\circ d_{\FTGK}=0$.  Since the maps $f^0$ form a dg $\mathfrak{K}$-modular operad map, we know that $f^0_{(1,4)}\circ d_{\FTGK} = 0$. We also know that $f^\omega\circ d_{\FTGK} = 0$ by Lemma $\ref{normlem}$.  Thus $f_{(1,4)}\circ d_{\FTGK} =(f_{(1,4)}^0+f^\omega)\circ d_{\FTGK}=0$.	
\end{proof}

\begin{lemma} $f_{(1,4)}$ is a quasi-isomorphism.
\end{lemma}
\begin{proof}
	With respect to the direct sum decomposition
	\begin{equation*}
	\DM(1,4)\cong	\DM(1,4)^{s\neq 4}\oplus \DM(1,4)^{s=4},
	\end{equation*}
	the map $f_{(1,4)}$ may be written $f^0\oplus f^\omega$.  Indeed $f^0$ restricted to $\DM(1,4)^{s=4}$ could only not vanish on graphs labeled by dual fundamental classes, and so could only not vanish on $\DM(1,4)^{-2,-4}$.  But $f^0$ restricted to this component lands in $\Sigma^{-8}H_{2}(\op{M}_{1,4})=0$.

	Writing $\mathfrak{p}H_\ast(\op{M}_{1,4}) = \op{P}\oplus \mathbb{Q} \omega $, it suffices to show that the restrictions
	\begin{equation*}
	f^0\colon \DM(1,4)^{\neq 4} \to \op{P}\text{ and } 
	f^\omega\colon \DM(1,4)^{4} \to \mathbb{Q} \omega  
	\end{equation*}
	are both quasi-isomorphisms.  For the former, this follows from Lemma $\ref{hlem}$, since when $s\neq 4$ all homology classes are represented by graphs labeled by dual fundamental classes at each vertex. For the latter, combining Lemma $\ref{hlem}$ with Lemma $\ref{normlem}$ shows the induced map on homology $\mathbb{Q}\to \mathbb{Q}$ is non-zero, and hence is an isomorphism.
\end{proof}

The map $f_{(1,4)}$ will correspond to the corolla of type $(1,4)$ in the $\infty$-morphism $f$ constructed below.  To uniformize our notation we define $f_{(g,n)} := f^0_{g,n}$ for $(g,n)<(1,4)$.

\section{Construction of $(1,4)$-truncated $\DM\stackrel{\sim}\rightsquigarrow \mathfrak{p}H_\ast(\op{M})$}
In this section we give a complete and explicit description of the $(1,4)$-truncated weak $\mathfrak{K}$-modular operad structures on $\mathfrak{p}H_\ast(\op{M})$ whose homotopy type coincides with $\DM$.  We then construct an explicit $\infty$-quasi-isomorphism between them extending the maps $f_{(g,n)}$ above.

\subsection{Higher Operations}

\begin{definition}\label{opsdef} Fix $\mathsf{e},\mathsf{w} \in \mathbb{Q}$ with $\mathsf{e}\neq 0$.  Let $\mathfrak{p}h_\ast^{\mathsf{e,w}}(\op{M})$ be the (1,4)-truncated weak $\mathfrak{K}$-modular operad characterized by the following properties:
\begin{itemize}
	\item Its underlying dg $\mathfrak{K}$-modular operad is the (1,4)-truncation of $\mathfrak{p}H_\ast(\op{M})$.
	\item Its only non-zero higher Massey products are given by $\gamma \in \Gamma(1,4)_2$ and supported on $\mathfrak{p}H_0(\op{M})$.
	\item Its operation indexed by a leaf labeling of $\delta_{\bullet}$ is $S_4$ invariant and defined to be $\mu_{\bullet}(1) = c_{\bullet}(\mathsf{e},\mathsf{w})\omega$, where $\omega$ is as defined in Definition $\ref{omegadef}$ and where the coefficients $c_{\bullet}(\mathsf{e,w})$ are defined in the following table: 
\end{itemize} 
\begin{center}
	\begin{tabular}{c|c|c|c|c|c|c|c|c|c}
	$\mu_{\bullet}$ &	$\mu_{2,2}$ & $\mu_{2,3}$ & $\mu_{3,4}$ & $\mu_{2,4}$ & $\mu_{0,4}$& $\mu_{0,3}$ & $\mu_{0,2}$ & 	$\mu_{\alpha}$ & $\mu_{\beta}$ \\ \hline
$c_{\bullet}(\mathsf{e,w})$ & $12\mathsf{e}$	 & $-4\mathsf{e}$ & $6\mathsf{e}$	 & $-2\mathsf{e}$	 & $\mathsf{e}+6\mathsf{w}$	 & $\mathsf{e}+3\mathsf{w}$ & $\mathsf{w}$	& $-3\mathsf{w}$  &	$-2\mathsf{e}-4\mathsf{w}$ \\
	\end{tabular}   		
\end{center}

\end{definition}
We often drop the word ``truncated'' when referring to $\mathfrak{p}h_\ast^{\mathsf{e,w}}(\op{M})$; recall it may be thought of as a non-truncated $\mathfrak{K}$-modular operad by extension by $0$.

\begin{lemma} As defined above, $\mathfrak{p}h_\ast^{\mathsf{e,w}}(\op{M})$ is a weak $\mathfrak{K}$-twisted modular operad.
\end{lemma}
\begin{proof}
	We need only check the differential condition.  But having truncated above (1,4) and the fact that the internal differential is $0$ make the condition vacuous.  
\end{proof}

\subsection{Uniqueness}
\begin{lemma}\label{unique}  For any pairs of scalars $(\mathsf{e,w})$ and $(\mathsf{e}^\prime,\mathsf{w}^\prime)$ with $\mathsf{e}\mathsf{e}^\prime \neq 0$, there exists an $\infty$-isomoprhism of weak $\mathfrak{K}$-modular operads 
	\begin{equation*}
	g\colon \mathfrak{p}h_\ast^{\mathsf{e},\mathsf{w}}(\op{M}) \stackrel{\cong}\rightsquigarrow  \mathfrak{p}h_\ast^{\mathsf{e}^\prime,\mathsf{w}^\prime}(\op{M})
	\end{equation*}
which we construct explicitly.
\end{lemma}

\begin{proof}  We first observe that the $\mathfrak{K}$-modular operads associated to the pairs $(\mathsf{e},0)$ and $(\mathsf{e}^\prime,0)$ are actually related by a strict isomorphism which takes $\omega\mapsto \mathsf{e} \omega/\mathsf{e}^\prime$ and which is the identity on the complementary summand.  By transitivity, it is thus sufficient to construct an $\infty$-morphism between a generic $(\mathsf{e},\mathsf{w})$ and $(\mathsf{e},\mathsf{w}^\prime)$.
	
So let us now assume that $\mathsf{e}=\mathsf{e}^\prime$.  We must define
	\begin{equation*}
g_\gamma \colon \mathfrak{p}h^{\mathsf{e},\mathsf{w}}_\ast(\op{M})(\gamma)_{Aut(\gamma)}\to \mathfrak{p}h^{\mathsf{e},\mathsf{w}^\prime}_\ast(\op{M}_{\ast_\gamma})
\end{equation*}
for each modular graph $\gamma$ of type $\ast_\gamma \leq (1,4)$.  We define $g_\gamma=0$ unless $\gamma$ is a corolla, in which case we define it to be the identity map, or if $\gamma=\mathsf{D}$ the loop with four legs (Figure $\ref{fig:141graphs}$), in which case we define $g_\gamma = g_\mathsf{D}$ as follows.
	
We identify $\mathfrak{p}h^{\mathsf{e},\mathsf{w}}_\ast(\op{M})(\mathsf{D})_{Aut(\mathsf{D})} \cong \Sigma^{-6}  H_\ast(\op{M}_{0,6})_{S_2}$, where $S_2$ acts by permuting $5$ and $6$ by convention.  Then $g_\mathsf{D}$ is equivalent to a family of linear maps
$H_n(\op{M}_{0,6})_{S_2} \to \Sigma^{-8} H_{n+2}(\op{M}_{1,4})$. 
We define these maps to be $0$ except when $n=1$, in which case we define
\begin{equation*}
g_\mathsf{D}\colon H_1(\op{M}_{0,6})_{S_2} \to H_3(\op{M}_{1,4})
\end{equation*}
(omitting suspensions by abuse of notation) as follows.

The 9 dimensional $S_6$-module $H_1(\op{M}_{0,6})$ is spanned by cyclic operadic compositions of $H_0(\op{M}_{0,i})$ for $i=3,4$.  These compositions can be represented by one edged trees with 6 leaves labeled $1\cdc 6$.  We view such a tree as rooted, with root taking label $6$; call the root vertex $a$ and the other vertex we call $b$.  There are $25$ such trees, we denote them by:
\begin{itemize}
	\item $X_{i,j}$ -- 
	$a$ is adjacent to legs labeled $\{i,j,6\}$, for $1\leq i< j \leq 5$.
	\item $Y_{i,j}$ --
		$b$ is adjacent to legs labeled $\{i,j\}$, for $1\leq i< j \leq 5$.
	\item $Z_{i}$ -- 
	$a$ is adjacent to legs labeled $\{i,6\}$, for $1\leq i \leq 5$.
\end{itemize}

We will define $g_\mathsf{D}$ to be the lift to $S_2\cong S_{\{5,6\}}$ coinvariants of the map sending the spanning set above to the following multiples of $\omega\in H_3(\op{M}_{1,4})$:
\begin{center}
	\begin{tabular}{r|c|c|c|c|c|c}
		Vector $\in H_1(\op{M}_{0,6})$ & $X_{i,5}$ & $X_{i,j <5}$   & $Y_{i,5}$   & $Y_{i,j<5}$   & $Z_5$ & $Z_{i<5}$                        \\ \hline
		Coefficient of $\omega$	& $3(\mathsf{w} - \mathsf{w}^\prime)$ & $2(\mathsf{w}^\prime - \mathsf{w})$ & $ 3(\mathsf{w}^\prime - \mathsf{w})/2$ & $(\mathsf{w} - \mathsf{w}^\prime)$ & $6(\mathsf{w} - \mathsf{w}^\prime)$ & $3(\mathsf{w}^\prime - \mathsf{w})/2$
	\end{tabular}
\end{center}

To check $g_\mathsf{D}$ is well defined we must verify that relations between compositions vanish.  These relations were derived and explicitly presented in \cite[Theorem 4.5]{Getgrav}. They are:
\begin{equation*}
 \sum_{1\leq i<j\leq 5} Y_{i,j} =0 \ \text{ and } \
 Z_i = Y_{j,k} + Y_{j,l} + Y_{j,m} + Y_{k,l} + Y_{l,m} + Y_{k,m}
 \ \text{ and } \ 
  X_{i,j} = Y_{k,l} + Y_{l,m} + Y_{k,m},
\end{equation*}
where $\{i,j,k,l,m\} = \{1,2,3,4,5\}$. 
 This verification is immediate. 
The fact that this map is $S_4$-invariant is also immediate.  
The fact that this map lifts to $S_{\{5,6\}}$-coinvariants is seen by observing $(56) Y_{i,5} = Z_i$, $(56) X_{i,j<5} = X_{l,k<5}$ and $(56)$ fixes $X_{i,5}$,  $Y_{i,j<5}$ and $Z_5$.

This completes the definition of the putative $\infty$-isomorphism $g$.  We now verify that it is in fact an $\infty$-morphism.  This amounts to checking the differential condition given in Equation $\ref{dcond0}$, which in this case reads:
\begin{equation}\label{dcond}
0 = \partial (g_\gamma)   = \sum_{N \text{ on } \gamma} (g_{\gamma/N}\circ_N \mu_N ) - (\sum_{V(\gamma) = \sqcup V(N_i)} \mu^\prime_{\gamma/\cup N_i} \circ \tensor g_{N_i})
\end{equation}
for each modular graph $\gamma$.  Here we've written $\mu$ (resp. $\mu^\prime$) for the operations on the source (resp. target).  They're degree -1 maps running:
	\begin{equation*}
	\mu_\gamma, \mu^\prime_\gamma \colon \mathfrak{p}H_\ast(\op{M})(\gamma)_{Aut(\gamma)}\to \mathfrak{p}H_\ast(\op{M}_{\ast_\gamma}).
	\end{equation*}
They coincide (commute with the identity) if $\gamma$ has one edge, but not if it has two.  

We proceed to check Equation $\ref{dcond}$.  First suppose $\gamma$ has one edge.   
	The only non-zero terms in the differential are $id\circ \mu$ and $\mu^\prime \circ id$, which cancel.  Any other term, i.e. a term involving $g_\mathsf{D}$ would also necessarily involve the internal differential which is $0$.
	
	When $\gamma$ has three or more edges, each term in the differential vanishes. 
	Indeed, given such a $\gamma$, for $\mu_N$ to not vanish would require either that $N$ has one edge, in which case $g_{\gamma/N}=0$, or $N$ is of type $(1,4)$, in which case $\gamma/N$ has a vertex of type $(1,4)$ and so $g_{\gamma/N}=0$. Hence each $g_{\gamma/N}\circ_N \mu_N=0$. On the other hand, each term of the form $\mu^\prime_{\gamma/\cup N_i} \circ \tensor g_{N_i}$ is easily seen to vanish, since each $\mu^\prime_{\gamma/\cup N_i}$ vanishes anytime $N_i$ is of type $(1,4)$ or if $\gamma/\cup N_i$ has three or more edges.

So it remains to check Equation $\ref{dcond}$ for $\gamma$ having two edges.  In this case, the differential terms include $\mu^\prime_\gamma\circ id$ and $id \circ \mu_\gamma$ which don't commute in general. 
Other terms of the form $\mu^\prime_{\gamma/\sqcup N_i}\circ \tensor g_{N_i}$ are $0$ since some $N_i$ would be of type $< (1,4)$.

The remaining terms in the differential are two terms of the form $g_{\gamma/N}\circ\mu_N$, one for each nest $N$ on $\gamma$.  By definition of $g$, such a term can only be non-zero if $\gamma/N = \mathsf{D}$. 
By inspection of Figure $\ref{fig:14graphs}$, this happens precisely in the 5 out of 9 cases in which $\mu^\prime_\gamma\circ id- id \circ \mu_\gamma  \neq 0$.
It remains to verify, in these five cases, that $\mu^\prime_\gamma\circ id- id \circ \mu_\gamma = \sum g_\mathsf{D}\circ\mu_N$.  The sum is over those nests with $\gamma/N =\mathsf{D}$; the sum has two terms when $\gamma$ has parallel edges and one term otherwise.

This equation is verified in the following table.   Here $\gamma$ is leg labeled, so the stable graph $\hat{N}$ associated to $N$ inherits a labeling of four of its legs by 1,2,3,4.   We then choose a labeling of the legs forming a loop in $\gamma/N$ by the set $\{5,6\}$.  Since $\mu_N$ lands in $S_{\{5,6\}}$-coinvariants, the result is independent of this choice and we write [-] for coinvariants.
\begin{center}
	\begin{tabular}{r|c|c|c|c|c}
		stable graph $\gamma$ & $\delta_{0,2}$ & $\delta_{0,3}$ & $\delta_{0,4}$ & $\delta_{\alpha}$ & $\delta_{\beta}$ \\ \hline
		$\mu^\prime_\gamma\circ id- id \circ \mu_\gamma$ 
		&$ (\mathsf{w}-\mathsf{w}^\prime) \omega$ &$ 3(\mathsf{w}-\mathsf{w}^\prime) \omega$ &$ 6(\mathsf{w}-\mathsf{w}^\prime) \omega$ &$ -3(\mathsf{w}-\mathsf{w}^\prime) \omega$ &$ -4(\mathsf{w}-\mathsf{w}^\prime) \omega$ \\
		$\mu_N$ & $[Y_{i,j<5}]$  &   $[X_{i,5}]$   & $[Z_5]$    &  $[Y_{i,5}],[Z_{i<5}]$  & $[X_{i,j <5}], [X_{l,k <5}]$  \\
		$\sum_Ng_\mathsf{D}\circ \mu_N$ & 1 & 3 & 6 & $-3/2-3/2$ & $-2-2$
	\end{tabular}
\end{center}
omitting the factor of $(\mathsf{w}-\mathsf{w}^\prime)\omega$ in each entry of the last row due to space considerations.
\end{proof}

\subsection{Higher Morphisms}
We now extend the family of quasi-isomorphisms $f_{(g,n)}$, constructed above, to a  $(1,4)$-truncated $\infty$-quasi-isomorphism
\begin{equation*}
f\colon \DM \rightsquigarrow \mathfrak{p}h^{\mathsf{e,w}}_\ast(\op{M}).
\end{equation*}
For this we must define a map $f_\gamma$ for each $\gamma \in \Gamma(g,n)_{\geq 1}$, and we define this map to be $0$ unless $\gamma \in \Gamma(1,4)_1$.    
 The image of such an $f_\gamma$ will be supported on the span of $\omega$, and thus supported on total degree $-5$.  This means that $f_\gamma$ will be an equivariant map landing in $S_4$-invariants, and hence lifts to $S_4$-coinvariants.  Thus it remains to define 
\begin{equation}\label{oneedge}
\left[ \ds\bigoplus_{\gamma\in \Gamma(1,4)_1} \ \ \FTGK(H_\ast(\overline{\op{M}}))(\gamma)^{-5}_{Aut(\gamma)}\right]_{S_4}\to \Sigma^{-8}H_3(\op{M}_{1,4}).
\end{equation}

The source of the map in Equation $\ref{oneedge}$ is 15 dimensional.  It has a basis
whose elements are indexed by stable (1,4)-graphs $\delta_\bullet$ with two edges, along with a choice of a distinguished edge.    The indexing is given by labeling the vertices of the graph $\delta_\bullet$ with dual fundamental classes, and then nesting the distinguished edge $e$ so that $\gamma = \delta_\bullet/e$.  In particular the vertex of $\gamma$ corresponding to $e$ will thus be labeled with an element of $\FTGK(H_\ast(\overline{\op{M}}))$ having one edge.  The nine such graphs pictured in Figure $\ref{fig:14graphs}$ each contribute two such vectors, except $\delta_{2,2},\delta_{\beta}$ and $\delta_{\alpha}$ which contribute one each.

We denote this basis as follows. The direct sum in Equation $\ref{oneedge}$ splits over the underlying stable graphs. Using the notation given in Figure $\ref{fig:141graphs}$, we denote the stable (1,4)-graphs with one edge by $\mathsf{A}, \mathsf{B},\mathsf{C}$ and $\mathsf{D}$.  
For the summands corresponding to $\mathsf{A}, \mathsf{B}$ and $\mathsf{C}$, the basis vectors correspond to the vectors $\lambda_i$ for $1\leq i \leq 10$ -- each corresponds to a choice of graph in Figure $\ref{fig:14graphs}$ along with a distinguished edge and having vertices labeled by dual fundamental classes.  We denote these corresponding basis vectors by $\Lambda_i$.  On the summand corresponding to $\mathsf{D}$ we have five such vectors whose nested graphs are pictured in Figure $\ref{fig:24}$.  We denote these basis vectors by $\xi_j$ for $1\leq j \leq 5$.

\begin{figure}
	\centering
	\includegraphics[scale=.85]{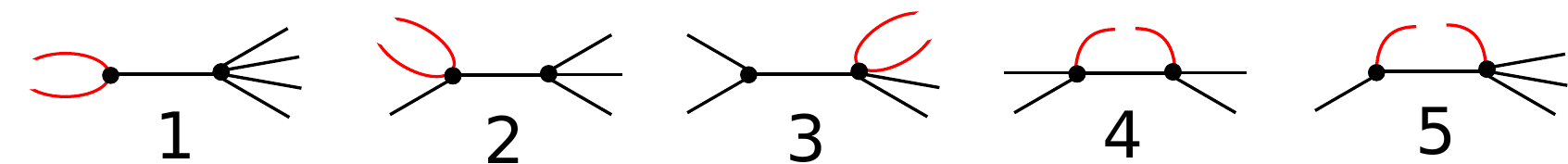}
	\caption{Basis vectors $\xi_j$ for $1 \leq j \leq 5$ are defined by labeling the vertices with dual fundamental classes, and joining the two curved legs colored red.  The pictured (black) edge is part of the nested graph which labels the lone vertex of $\mathsf{D}$ and so contributes $-1$ to the degree.}
	\label{fig:24}
\end{figure}

Using the direct sum decomposition, we now define the map in Equation $\ref{oneedge}$ as the direct sum of maps $f_\mathsf{A}, f_\mathsf{B}, f_\mathsf{C}$ and $f_\mathsf{D}$: 
\begin{align}\label{Lambdadef}
f_\mathsf{A}(\Lambda_i)=-\bar{f}^\omega(\lambda_i) & \text{ for } 1\leq i \leq 3 \\ \nonumber
f_\mathsf{B}(\Lambda_i)=-\bar{f}^\omega(\lambda_i) & \text{ for } 4\leq i \leq 6 \\ \nonumber
f_\mathsf{C}(\Lambda_i)=-\bar{f}^\omega(\lambda_i) & \text{ for } 7\leq i \leq 10
\end{align} 
and $f_\mathsf{D}(\xi_j) = c_j\omega$, where $c_j$ is given in the following table:
\begin{equation}
	\begin{tabular}{c|c|c|c|c|c} \label{fD}
		Vector &  $\xi_1$ & $\xi_2$ & $\xi_3$ & $\xi_4$ & $\xi_5$ \\ \hline
		coefficient $c_j$ & $\mathsf{e}/2 + 6\mathsf{w}$ & $\mathsf{e}/2 + 3\mathsf{w}$ & $\mathsf{w}$ & $-\mathsf{e}-\mathsf{2w}$ & $-3\mathsf{w}/2$
	\end{tabular}   		
\end{equation}
We emphasize that for a modular (1,4)-graph $\gamma$, the operation $f_\gamma$ is given by first passing to $S_4$ coinvariants and then applying the map above which corresponds to the underlying stable graph of $\gamma$.   This completes the definition of $f_\gamma$ and hence of $f$.  

\subsection{Differential condition}
Having defined the putative $\infty$-quasi-isomorphism $f$, we check the (remaining) differential conditions.  Denote the generating operations for the (weak) $\mathfrak{K}$-modular operad structure on the source of $f$ by $\nu$:
\begin{equation*}
\nu_\gamma\colon \FTGK(H_\ast(\overline{\op{M}}))(\gamma)_{Aut(\gamma)}\to \FTGK(H_\ast(\overline{\op{M}}))(\ast_\gamma)
\end{equation*}
In particular $\nu_\gamma$ is $0$ unless the graph $\gamma$ has 1 edge.  Denote, as above, the weak $\mathfrak{K}$-modular operad structure on the target by $
\mu_\gamma\colon \mathfrak{p}h^{\mathsf{e,w}}_\ast(\op{M})(\gamma)_{Aut(\gamma)}\to \mathfrak{p}h^{\mathsf{e,w}}_\ast(\op{M}_{\ast_\gamma})$.  We must verify the differential condition of Equation $\ref{inftydcond}$. 

If $\gamma$ is a graph of type less than $(1,4)$, this follows from the fact that the $(1,3)$-truncation of $f$ is a strict $\mathfrak{K}$-modular operad map, so we may restrict attention to graphs $\gamma$ of type $(1,4)$.  If such a $\gamma$ has 0 edges, the differential condition was verified in Corollary $\ref{dgcor}$.  If such a $\gamma$ has 3 or more edges, each term in the differential condition will vanish.  Thus it remains to consider the case when $\gamma$ is a modular $(1,4)$-graph having 1 or 2 edges.

{\bf Case 1: $\gamma \in \Gamma(1,4)_1$.}   Suppose $\gamma$ is a modular $(1,4)$-graph having 1 edge.  Any non-empty nest $N$ on $\gamma$ must also have 1 edge, hence $\gamma/N$ is the $(1,4)$-corolla.  Thus the first sum in Equation $\ref{inftydcond}$ is given by $\sum_{|N|=1} (f_{\gamma/N}\circ_N \nu_N ) =f_{(1,4)}\circ \nu_\gamma$. 

Consider now the second sum in Equation $\ref{inftydcond}$.  Since the internal differential in $\mathfrak{p}H_\ast(\op{M})$ is zero, the only non-zero term is $\mu_\gamma \circ \tensor f_\ast$, where $\tensor f_\ast$ is a product of corolla maps applied to each vertex of $\gamma$.  Recall the definition $f_{(1,4)}= f^0+f^\omega$.  Since $f^0$ is part of a strict $\mathfrak{K}$-modular operad map and $\gamma$ has one edge we know $f^0 \circ \nu_\gamma = \mu_\gamma \circ \tensor f_\ast$.
Therefore the differential condition of Equation $\ref{inftydcond}$ reduces to: 
\begin{equation}\label{d1}
f_\gamma \circ d_{\FTGK} = 
-(f^0+f^\omega) \circ \nu_\gamma + \mu_\gamma \circ f_\ast = - f^\omega \circ \nu_\gamma
\end{equation}
Since these maps both land in $S_4$-invariants, it's enough to verify their lifts to $S_4$-coinvariants coincide.

The maps in Equation $\ref{d1}$ are only non-zero when evaluated on $\DM(\gamma)^{0,-4}$ for $\gamma \in \Gamma(1,4)_1$.  Here we consider $\DM(\gamma)$ to be bigraded by the number of internal edges and the sum of the internal vertex degrees; the edge of $\gamma$ has degree $0$.
The $\mathfrak{K}$-twisted modular operad structure maps $\nu_\gamma$ of $\DM$ combine to give us a 
degree $-1$ map:
 \begin{equation*}
 \ds\bigoplus_{\gamma\in \Gamma(1,4)_1} \FTGK(H_\ast(\overline{\op{M}}))(\gamma)_{Aut(\gamma)}^{0,-4} \stackrel{\nu} \longrightarrow  \FTGK(H_\ast(\overline{\op{M}}))(1,4)^{-1,-4}.
 \end{equation*}
This map is tautologically an $S_4$-equivariant isomorphism.  
 Let $\bar{\nu}$ denote the induced isomorphism on $S_4$-coinvariants.

\begin{lemma}\label{lemftd}  The Feynman transform differential induces the differential
	\begin{equation*}
\left[\bigoplus_{\gamma \in \Gamma(1,4)_1}
\FTGK(H_\ast(\overline{\op{M}}))(\gamma)_{Aut(\gamma)}^{0,-4} 	\right]_{S_4}
	\stackrel{\hat{d}_{\FTGK}}\to	
	\left[\bigoplus_{\gamma \in \Gamma(1,4)_1} \FTGK(H_\ast(\overline{\op{M}}))(\gamma)^{-1,-4}_{Aut(\gamma)}\right]_{S_4}
	\end{equation*}
given in the basis	$\{\lambda_i \ | \ 1\leq i\leq 14\}$ by:
\begin{equation*}
\begin{tabular}{r|| c | c | c | c | c}
$\lambda_i$ &	$1\leq i \leq 10$ & $\lambda_{11}$ & $\lambda_{12}$ & $\lambda_{13}$ & $\lambda_{14}$ \\ \hline
 $\hat{d}_{\FTGK}(\bar{\nu}^{-1}(\lambda_i))$	& $\Lambda_i$ & $\xi_1+4\xi_5$ &   $\xi_2+2\xi_5$ &  $\xi_3+2/3\xi_5$ &  $2\xi_4-8/3\xi_5$
\end{tabular}.
\end{equation*}
\end{lemma}

\begin{proof}	 
	
		Each $\bar{\nu}^{-1}(\lambda_i)$ specifies a graph $\gamma$ 	
	with exactly one vertex $v$ not labeled by a dual fundamental class.
	The induced differential $\hat{d}_{\FTGK}$ is given by applying the Feynman transform differential at the vertex $v$.  By Lemma $\ref{Qcomplex}$ there are no relations among possible one edged compositions labeling $v$ if the graph $\gamma$ is of type  $\mathsf{A}$, $\mathsf{B}$, or $\mathsf{C}$.  Thus for $1\leq i \leq 10$, the vector $\hat{d}_{\FTGK}(\bar{\nu}^{-1}(\lambda_i))$ is given by a stable $(1,4)$-graph $\eta$ labeled by dual fundamental classes, and having a distinguished edge $e$ (from applying the Feynman transform differential) satisfying $\eta/e = \gamma$. This was exactly the definition of $\Lambda_i$ above.
	
	When $i \geq 11$, we have $\gamma = \mathsf{D}$ which has a unique vertex whose legs are partitioned in to a set of size 4 and a set of size 2, which we write as $\{1,2,3,4\}$ and $\{5,6\}$ by convention.  Thus the differential $\hat{d}_{\FTGK}$ is induced by taking $S_4$ coinvariants of the following commutative diagram: 
		\begin{equation*}
	\xymatrix{ &	H_4(\overline{\op{M}}_{0,6})^\ast \ar[d]^{Aut(\mathsf{D})} \ar[r]^{d_{\FTGK} \ \ \ \ \ \ \ } \ar[ld]_{\circ_{56}} &  \ar[d]^{Aut(\mathsf{D})}	\DM(0,6)^{-1,-4}  \\
		\DM(1,4)^{-1,-4}  & \ar[l]_{ \ \ \nu|_{\mathsf{D}}} \DM(\mathsf{D})_{Aut(\mathsf{D})}^{0,-4}	 \ar@{.>}[r]^{\hat{d}_{\FTGK} \ \ }  & \DM(\mathsf{D})_{Aut(\mathsf{D})}^{-1,-4} }
	\end{equation*}
	This map is in turn specified by dualizing the cyclic operadic composition map
\begin{equation*}
[\bigoplus_{T\in \Gamma(0,6)} H_\ast(\overline{\op{M}})(T)]^{S_4} \to H_4(\overline{\op{M}}_{0,6})^{S_4}.
\end{equation*}
The source of this map is 6 dimensional: it has a basis given by trees as in Figure $\ref{fig:24}$ along with a labeling of the red legs by 5 and 6.  Composing along these trees produces 6 vectors in $
H_4(\overline{\op{M}}_{0,6})^{S_4}$  which we call $\rho_i$, corresponding to the trees labeled $i=1,2,3,4$, and $\rho_5^1, \rho_5^2$ corresponding to the two labelings of the fifth graph in Figure $\ref{fig:24}$.  Calculating with the WDVV equation, these 6 compositions are related by the formula 
$\rho^1_5+\rho^2_5= 4 \rho_{1} +2\rho_2+2/3\rho_3 - 4/3\rho_4. $

Fix the basis $\rho_i,\rho_5^1$ and dualize. In the Feynman transform $\DM$ we have $\circ_{56}(\rho_1^\ast)=\lambda_{11}$, 
$\circ_{56}(\rho_2^\ast)=\lambda_{12}$,
$\circ_{56}(\rho_3^\ast)=\lambda_{13}$, and
$\circ_{56}(\rho_4^\ast)=\lambda_{14}/2$.  This factor of 1/2 comes from the fact that $\lambda_{14}$ was defined by dualizing an $S_4$ equivariant vector formed by the three ways to label the graph $\delta_\beta$, where as $\rho_4$ is the $S_4$ equivariant vector formed by the 6 ways to label the black legs in the fourth tree in Figure $\ref{fig:24}$.  Gluing the red legs together produces each labeling of $\delta_\beta$ twice.  It remains to simply chase through the diagram to arrive at the stated result. \end{proof}

\begin{corollary}  The one-edge differential condition (Equation $\ref{d1}$) is satisfied. 
\end{corollary}
\begin{proof}  
	Let $\bar{f}_{\gamma}$ denote the lift of $f_\gamma$ to $S_4$-coinvariants.  By Equation $\ref{d1}$,  it is sufficient to show $\bar{f}_{\gamma}(d_{\FTGK}(\bar\nu^{-1}(\lambda_i))) = -\bar{f}^\omega (\lambda_i)$. 	
	This follows by combining Lemma $\ref{lemftd}$ with Equation $\ref{t}$.
\end{proof}

{\bf Case 2: $\gamma \in \Gamma(1,4)_2$.} We now verify Equation $\ref{inftydcond}$ holds when $\gamma$ has two edges. 
The differential terms are as follows.  In the first sum each non-empty nest must have 1 edge, hence $\gamma/N$ has 1 edge.  So we get two terms, one for each nest, of the form $f_{\gamma/N}\circ\nu_N$.  
In the second sum, we have the non-zero term $\mu_\gamma \circ \tensor f_\ast.$ 
Any other term in this sum would involve $f_{\gamma^\prime}$ where $\gamma^\prime$ is a proper subgraph of $\gamma$, and so has type less than $(1,4)$ with one edge, but such operations were defined to be zero.

Therefore, if $\gamma$ is a modular $(1,4)$ graph with edges $\text{I}$ and $\text{II}$ we must verify the constraint:
\begin{equation*}
\mu_\gamma \circ \tensor f_\ast = f_{\gamma/\text{I}} \circ \nu_\text{I} + f_{\gamma/\text{II}} \circ \nu_\text{II}.
\end{equation*} 
As above, since these functions are supported on $S_4$ invariants, it's enough to verify their passages to $S_4$ coinvariants coincide.  Compiling information from above (Equations $\ref{t}$, $\ref{Lambdadef}$, $\ref{fD}$ and Definition $\ref{opsdef}$), we verify this constraint in the following table:

\begin{center}
	\begin{tabular}{r|c||c|c|c|c}
		higher operation & coefficient & apply $\nu_\text{I}$ & apply $f_{\gamma/\text{I}}$ &	apply $\nu_\text{II}$ & apply $f_{\gamma/\text{II}}$  \\ \hline
		$\mu_{2,2}$ & $12\mathsf{e}$ & $\Lambda_{10}$ & $6\mathsf{e}$  & $\Lambda_{10}$ & $6\mathsf{e}$	\\
		$\mu_{2,3}$ & $-4\mathsf{e}$ & $\Lambda_9$ & $-2\mathsf{e}$  & $\Lambda_6$& $-2\mathsf{e}$	\\
		$\mu_{3,4}$ & $6\mathsf{e}$ & $\Lambda_5$ &  $3\mathsf{e}$ & $\Lambda_1$ & $3\mathsf{e}$	\\
		$\mu_{2,4}$ & $-2\mathsf{e}$ & $\Lambda_7$ & $-\mathsf{e}$  & $\Lambda_2$ & $-\mathsf{e}$	\\
		$\mu_{0,4}$ & $\mathsf{e}+6\mathsf{w}$ & $\Lambda_3$ & $\mathsf{e}/2$  &  $\xi_1$& $\mathsf{e}/2+6\mathsf{w}$	\\
		$\mu_{0,3}$ & $\mathsf{e}+3\mathsf{w}$ & $\Lambda_4$  & $\mathsf{e}/2$  & $\xi_2$ & $\mathsf{e}/2 + 3\mathsf{w}$	\\
		$\mu_{0,2}$ & $\mathsf{w}$ & $\Lambda_8$ &  $0$ & $\xi_3$ & $\mathsf{w}$	\\
		$\mu_{\beta}$ & $-2\mathsf{e}-4\mathsf{w}$ &   $\xi_4$ &  $-\mathsf{e}-2\mathsf{w}$ & $\xi_4$ & $-\mathsf{e}-2\mathsf{w}$	 \\
				$\mu_{\alpha}$ & $-3\mathsf{w}$ & $\xi_5$  & $-3\mathsf{w}/2$  & $\xi_5$	&  $-3\mathsf{w}/2$
	\end{tabular}   		
\end{center}

This concludes the proof of:

\begin{theorem}  As constructed above, 
\begin{equation*}
f\colon\DM\stackrel{\sim}\rightsquigarrow \mathfrak{p}h^{\mathsf{e,w}}(\op{M})
\end{equation*}
	is an $\infty$-quasi-isomorphism of (1,4)-truncated $\mathfrak{K}$-modular operads.
\end{theorem}


\end{document}